\documentclass[letterpaper,10pt]{amsart}
\usepackage{indentfirst} % Indent first paragraph after section header
\usepackage{amssymb}
\usepackage{amsmath}
\usepackage{mathabx} % makes many symbols (as $\leq$) more beautiful
\usepackage{amsthm}
\usepackage{thmtools}
\usepackage{bbm}             %to use $\mathbbm{1}$
\usepackage[colorlinks=true]{hyperref}  % Links in the pdf
\usepackage[usenames,dvipsnames]{xcolor} % With XCOLOR, printed quality is good, (but with COLOR it's bad)
\usepackage{comment}

\newtheorem{theorem}{Theorem}[section]
\newtheorem{definition}[theorem]{Definition}
\newtheorem{proposition}[theorem]{Proposition}
\newtheorem{corollary}[theorem]{Corollary}
\newtheorem{lemma}[theorem]{Lemma}

\newtheorem{remark}[theorem]{Remark}

\newtheorem*{theorem-non}{Theorem}

\declaretheorem[name=Acknowledgements,numbered=no]{ack}

\theoremstyle{definition}

\newcommand{\R}{\mathbb{R}}
\newcommand{\Z}{\mathbb{Z}}
\newcommand{\N}{\mathbb{N}}

\renewcommand{\P}{\mathbb{P}}

\newcommand{\cD}{\mathcal{D}}

\def\phi{\varphi}
\def\R{{\mathbb R}}

\def\N{{\mathbb N}}
\def\Z{{\mathbb Z}}

\def\P{{\mathcal P}}

\def\F{{\mathcal F}}

\def\diam{\mbox{\rm diam} }

\def\le{\leqslant}
\def\ge{\geqslant}

\def\F{\mathcal{F}}

\begin{document}

\title{Escape of mass and entropy for geodesic flows}
\date{\today}

\author[F. Riquelme]{Felipe Riquelme}
%\thanks{F.R. was supported by Programa de Cooperaci\'on Cient\'ifica Internacional CONICYT-CNRS c\'odigo PCCI 14009.}
\address{IMA, Pontificia Universidad Cat\'olica de Valpara\'iso, Blanco Viel 596, Cerro Bar\'on, Valpara\'iso, Chile.}

\email{friquelme.math@gmail.com}
\urladdr{\url{https://friquelme-math.com/}}

\author[A. Velozo]{Anibal Velozo}  \address{Princeton University, Princeton NJ 08544-1000, USA.}
\email{avelozo@math.princeton.edu}

\begin{abstract}
In this paper we study the ergodic theory of the geodesic flow on negatively curved geometrically finite manifolds. We prove that the measure theoretic entropy is upper semicontinuous when there is no loss of mass. In case we are losing mass, the critical exponents of parabolic subgroups of the fundamental group have a significant meaning. More precisely, the failure of upper-semicontinuity of the entropy is determinated by the maximal parabolic critical exponent. We also study the pressure of positive H\"older continuous potentials going to zero through the cusps. We prove that the pressure map $t\mapsto P(tF)$ is differentiable until it undergoes a phase transition, after which it becomes constant. This description allows, in particular, to compute the entropy at infinity of the geodesic flow.
\end{abstract}

\maketitle

\section{Introduction}
Geodesic flows on negatively curved manifolds are one of the main examples of Anosov flows. As such, they have motivated  a lot of  research in partially hyperbolic dynamical systems during the last century. They have also been used as ground to test conjectures before jumping into a more general framework. From another point of view, they allow us to use ergodic techniques to understand metrics on our ambient manifold. Of course, it does not look like we are simplifying the work, long time behaviour is subtle and not very computable. Fortunately, after the work of many people, we have a reasonably good understanding of the dynamics in many cases (for instance if the manifold is compact or geometrically finite). This paper points mainly in the first direction. We investigate continuity properties of the entropy in noncompact Riemannian manifolds and discuss the phenomena of escape of mass.

%The entropy at infinity might be useful to say something about the isometry groups of our manifold (this facet will be pursued in future works).

This paper has two main goals. The first is to prove the uppersemicontinuity of the measure theoretic entropy in the geometrically finite case. This part of the work is based on \cite{MR3430275} where Einsiendler, Kadyrov and Pohl proved a very similar result in the case of finite covolume lattices in real rank 1 Lie groups. In our context the constants involved have a clear geometrical meaning. Assume $\widetilde{M}$ is a simply connected, pinched negatively curved manifold and $\Gamma$ a discrete, torsion free group of isometries of $\widetilde{M}$. For such  $\Gamma$ we define
$$\overline{\delta}_{\P}=\sup\{\delta_{\mathcal{P}}: \mathcal{P} \ \mbox{parabolic subgroup of} \ \Gamma\},$$
where $\delta_G$ is the critical exponent of a subgroup $G$ of isometries of $\widetilde{M}$. Let $M$ be the quotient Riemannian manifold $\widetilde{M}/\Gamma$. We say that a sequence $(\mu_n)$ of probability measures on $T^1M$ \textit{converges vaguely} to $\mu$ if $\lim_{n\to\infty}\int f d\mu_n= \int fd\mu$ whenever $f$ is a compactly supported function. With this topology, the space of invariant probability measures is not necessarily compact because of the loss of mass phenonema. In section \ref{escape} we prove

\begin{theorem}\label{thm:escape_mass} Let $(M,g)$ be a geometrically finite Riemannian manifold with pinched negative sectional curvature. Assume that the derivatives of the sectional curvature are uniformly bounded. Let $(\mu_n)$ be a sequence of $(g_t)$-invariant probability measures on $T^1M$ converging to $\mu$ in the vague topology. Then
$$\limsup_{n\to\infty} h_{\mu_n}(g) \leq \|\mu\|h_{\frac{\mu}{\|\mu\|}}(g)+(1-\|\mu\|)\overline{\delta}_{\P}.$$
\end{theorem}

In particular, if the sequence $(\mu_n)$ does not lose mass, then the classical upper semicontinuity result follows. We remark that we can not just cite works of Yomdin \cite{yo} and Newhouse \cite{n} because of the noncompactness of $T^1M$. The following corollary follows directly from the Theorem above.

\begin{corollary}\label{cor:mass_estimationbelow} Assume the hypothesis of Theorem \ref{thm:escape_mass}. Let $(\mu_n)$ be a sequence of $(g_t)$-invariant probability measures on $T^1M$ converging to $\mu$ in the vague topology. If for large enough $n$ we have $h_{\mu_n}(g)\geq c$, then
$$\|\mu\|\geq \frac{c-\overline{\delta}_{\P}}{\delta_\Gamma-\overline{\delta}_\P}.$$
In particular, if $c>\overline{\delta}_\P$, then any vague limit of $(\mu_n)$ has positive mass.
\end{corollary}

In \cite{IRV} the authors defined the entropy at infinity $h_\infty(g)$ of the geodesic flow as
$$h_\infty(g)=\sup_{\mu_n \rightharpoonup 0} \limsup_{n\to\infty} h_{\mu_n}(g).$$
This number is $h_{top}(g)/2$ on the cases treated in \cite{MR3430275}, in particular for the geodesic flow on noncompact finite volume hyperbolic surfaces. It was also checked to be $\overline{\delta}_\P$ in the case of extended Schottky manifolds by the use of symbolic dynamics in \cite{IRV}. We can now cover all the geometrically finite cases. Observe that under the hypothesis of Theorem \ref{thm:escape_mass} we get
$$h_\infty(g)\le \overline{\delta}_\P.$$
In fact, we have the equality $h_\infty(g)= \overline{\delta}_\P$ as consequence of the following result.
\begin{theorem}\label{thm:existence_sequence} Assume the hypothesis of Theorem \ref{thm:escape_mass}. Then there exists a sequence $(\mu_n)$ of ergodic $(g_t)$-invariant probability measures on $T^1M$ converging vaguely to 0 and such that
$$\lim_{n\to\infty} h_{\mu_n}(g) = \overline{\delta}_{\P}.$$
\end{theorem}
It worth mentioning that only when $h_{\infty}(g)<h_{top}(g)$ the result above becomes non trivial. In fact the authors do not known of any other case when a statement like that is true. It is shown in Section \ref{infinity} that this is not a general feature, having a big group of isometries prevent to have such a gap.

The second goal of the paper is to study the pressure map $t\mapsto P(tF)$ for H\"older continuous potentials going to zero through the cusps of the manifold. We get a fairly complete description of this map. Define $\mathcal{F}$ as the space of positive H\"older continuous potentials converging to zero through the cusps (see Definition \ref{def-pot}). For $F\in\mathcal{F}$ we have the following result.
\begin{theorem} \label{1} Let $M=\widetilde{M}/\Gamma$ be a geometrically finite Riemannian manifold with pinched negative sectional curvature. Assume that the derivatives of the sectional curvature are uniformly bounded. Then every potential $F\in\mathcal{F}$ verifies
\begin{enumerate}
\item[(1)] for every $t \in \R$, we have $P(t F) \geq \overline{\delta}_\P$
\item[(2)] the function $t\mapsto P(tF)$ has a horizontal asymptote at $-\infty$, and it verifies
$$ \lim_{t \to -\infty} P(tF)= \overline{\delta}_\P.$$
\end{enumerate}
Additionally, if $t':= \sup\left\{ t \leq 0 : P(tF)= \overline{\delta}_\P \right\}$, then
\begin{enumerate}
\item[(3)] for every $t>t'$ the potential $tF$ has a unique equilibrium measure, and
\item[(4)] the pressure function $t\mapsto P(tF)$ is differentiable in $(t',\infty)$, and it verifies
\begin{equation*}
P(tF)=
\begin{cases}
\overline{\delta}_\P & \text{ if } t < t'\\
\text{strictly increasing}  & \text{ if } t > t',
\end{cases}
\end{equation*}
\item[(5)] If $t<t'$ then the potential $tF$ has not equilibrium measure.
\end{enumerate}
\end{theorem}

The paper is organized as follows. In Section \ref{sec_1} we recall some facts about measure theoretical entropy and the ergodic theory of geodesic flows on negatively curved manifolds. In Section \ref{tech} we prove some technical results used in later sections. In Section \ref{escape} we prove Theorem \ref{thm:escape_mass}. In Section \ref{appli} we study the pressure of potentials in the family $\mathcal{F}$ and prove Theorem \ref{1}. In Section \ref{infinity} we discuss a case when $h_\infty(g)=h_{top}(g)$. Finally, in Section \ref{final} we discuss some final remarks and reprove a Theorem of Dal'bo, Otal and Peign\'e.

\begin{ack}
 We thanks G. Iommi for the stimulating collaboration \cite{IRV}, which is the starting point of this project. We also thanks B. Schapira for very useful discussions at the very beginning of this work. Finally, the second author would like to thanks to his advisor G. Tian for his constant support and encouragements.
\end{ack}

\section{Preliminaries}\label{sec_1}
In this section we recall information about entropy and geodesic flow on negatively curved manifolds. We also introduce notation used in following sections.

\subsection{Entropy}

%To quantify the complexity of a given dynamical system we need to pay attention on its nature.
We will be working with two different types of entropy, the measure theoretic entropy and the topological entropy. Each of them collect information about the complexity of the dynamical system (in an apparently different way, but not completely unrelated as we will see below). Indeed, for measure-preserving transformations on measure spaces we can define the notion of \emph{measure entropy} to quantify chaos. On the other hand, for continuous transformations on topological spaces we use the notion of \emph{topological entropy}. In this subsection we recall the standard definitions and state classical relations between them.\\

Let $(X,\mathcal{B},\mu)$ be a standard probability space. A countable (resp. finite) partition $\mathcal{P}$ of $X$ is a countable (resp. finite) collection $\{P_\alpha\in\mathcal{B} : \alpha\in\mathcal{A}\}$ indexed by $\mathcal{A}$ of countable (resp. finite) cardinality, such that $P_\alpha\cap P_\beta=\emptyset$ for every $\alpha\neq \beta$, and $\mu\left(\bigcup_{\alpha\in\mathcal{A}} P_\alpha\right)=1$. The entropy of $\mathcal{P}$ is the non-negative real number defined as
$$H_\mu(\mathcal{P})=-\sum_{\alpha\in\mathcal{A}} \mu(P_\alpha)\log\mu(P_\alpha).$$
Let $\mathcal{P}$ and $\mathcal{Q}$ be two countable (or finite) partitions of $X$. The partition $\mathcal{P}\vee\mathcal{Q}$ is the countable (or finite) partition defined as
$$\mathcal{P}\vee\mathcal{Q}=\{P\cap Q : P\in\mathcal{P} \quad\mbox{and}\quad Q\in\mathcal{Q}\}.$$
Observe that entropy has the following subadditive property: for $\mathcal{P}$ and $\mathcal{Q}$ countable (or finite) partitions, we always have
$$H_\mu(\mathcal{P}\vee\mathcal{Q})\leq H_\mu(\mathcal{P})+H_\mu(\mathcal{Q}).$$
Consider now a measurable transformation $T:X\to X$ preserving the measure $\mu$. Note that if $\mathcal{P}$ is a countable (resp. finite) partition of $X$, then for every $i\in\N$ the partition $T^{-i}\mathcal{P}$ is also a countable (resp. finite) partition of $X$. Denote by $\mathcal{P}^n$ the partition $\bigvee_{i=0}^{n-1}T^{-i}\mathcal{P}=\mathcal{P}\vee T^{-1}\mathcal{P}\vee...\vee T^{-n+1}\mathcal{P}$. The entropy of $T$ with respect to $\mathcal{P}$ is defined by
$$h_\mu(T,\mathcal{P})=\lim_{n\to\infty}\frac{1}{n}H_\mu(\mathcal{P}^n).$$
Note that the limit in the definition above always exists because the sequence $(H_\mu(\mathcal{P}^n))_n$ is subadditive, that is $H_\mu(\mathcal{P}^{m+n})\leq H_\mu(\mathcal{P}^m)+H_\mu(\mathcal{P}^n)$ for all $m,n\in\N$. Here we are using the fact that $T$ preserves $\mu$. The entropy of $T$ with respect to $\mu$, also called measure theoretical entropy of $T$, is then defined as
$$h_\mu(T)=\sup_{\mathcal{P}} h_\mu(T,\mathcal{P}),$$
where the supremum is taken over all countable (or finite) partitions $\mathcal{P}$ of $X$.\\

Suppose now that $X$ is a locally compact metrizable topological space and $T:X\to X$ is a continuous transformation. Fix a distance $d$ on $X$. For every $n\geq 1$, $r>0$, define the $(n,r)$-dynamical ball centered at $x\in X$ with respect to $d$ as
$$B^d_n(x,r)=\{y\in X : d(T^ix,T^i y)<r, \quad\mbox{for every}\quad 0\leq i \leq n-1\}.$$
It will be convenient to consider the metrics $d_n$ defined by $$d_n(x,y)=\max_{k\in\{0,...,n\}}\{d(T^kx,T^ky)\}.$$ The $(n,r)$-dynamical ball is just the $r$-ball with respect to the metric $d_n$. Since $T$ is continuous, every $(n,r)$-dynamical ball is an open subset of $X$. In particular, every compact set $K\subset X$ admits a finite $(n,r)$-covering, that is a finite covering by $(n,r)$-dynamical balls. Denote by $N^d(n,r,K)$ the minimal cardinality of a $(n,r)$-covering of $K$. The topological entropy $h^d(T)$ of $T$ with respect to the distance $d$ is defined as
$$h^d(T)=\sup_{K}\lim_{r\to 0}\limsup_{n\to\infty} \frac{1}{n}\log N^d(n,r,K),$$
where the supremum is taken over all compact subsets of $X$. Finally define the topological entropy of $T$ as
$$h_{top}(T)=\inf_d h^d(T),$$
where the infimum is taken over all the distances $d$ on $X$ generating the topology of $X$. Note that when $X$ is compact, all distances on $X$ are equivalent. In particular, the topological entropy $h_{top}(T)$ coincides with the entropy $h^d(T)$ for every distance $d$ compatible with the topology of $X$. We remark that instead of using the minimal number of a $(n,r)$-covering of $K$ we could have used the maximum number of $(n,r)$-separated points in $K$, that is points that are $r$-separated with respect to the metric $d_n$. By standard arguments both definitions coincide.\\

The measure theoretic entropy and topological entropy are related by the variational principle (see for instance \cite{Din} for the compact case and \cite{MR1348316} for the general one).

\begin{theorem}[Variational Principle]\label{thm:variational_principle} Let $X$ be a locally compact metrizable topological space. Then
$$h_{top}(T)=\sup_\mu h_\mu(T),$$
where the supremum is taken over all $T$-invariant probability measures on $X$.
\end{theorem}

The topological entropy is related to the cardinalities of $(n,r)$-coverings of compact subsets of $X$, we clearly do not need of any measure to count that. However, the variational principle shows a sort of intrinsic relation between cardinalities of $(n,r)$-coverings and measures. This relation seems to be more natural and evident after considering the following definition and theorem below. Let $\mu$ be a $T$-invariant probability measure and let $0<\delta<1$. Define $N^d_\mu(n,r,\delta)$ as the minimal cardinality of a $(n,r)$-covering of a subset of $X$ with measure larger than $1-\delta$. Note that $N^d_\mu(n,r,\delta)$ is always finite since $N^d_\mu(n,r,\delta)\leq N^d(n,r,K)$ for every compact set $K\subset X$ with measure $\mu(K)>1-\delta$.

\begin{theorem}[Katok]\label{thm:katok} Let $(X,d)$ be a complete locally compact metric space and $T:X\to X$ a continuous transformation. If $\mu$ is an ergodic $T$-invariant probability measure on $X$, then for every $0<\delta<1$, we have
\begin{equation}\label{eq:katok}
h_\mu(T)\leq \lim_{r\to 0}\liminf_{n\to\infty} \frac{1}{n}\log N^d_\mu(n,r,\delta).
\end{equation}
\end{theorem}

\begin{remark}
Katok proved in \cite{MR573822} that \eqref{eq:katok} is an equality for continuous transformations defined on compact metric spaces. For non-compact metric spaces his proof only gives the above inequality.
\end{remark}

\subsection{Geodesic flow on negatively curved manifolds}
\noindent \\
Let $(\widetilde{M},g)$ be a simply connected, complete, negatively curved Riemannian manifold satisfying the pinching condition $-b^2\le K_g\le -a^2$ for some real numbers $a$ and $b$ where $a>0$. Here $K_g$ stands for the sectional curvature of $g$. We moreover assume that the derivative of the sectional curvature is uniformly bounded. The classical example to have in mind is hyperbolic $n$-space, where the sectional curvature is constant equal to $-1$. As usual, we denote by $\partial_\infty \widetilde{M}$ to the Gromov boundary at infinity of $\widetilde{M}$. We will be interested in quotients of $\widetilde{M}$ by certain groups of isometries. Let $\Gamma$ be a discrete, torsion free subgroup of $Iso(\widetilde{X})$. Define $M=\widetilde{M}/\Gamma$ and observe that by hypothesis $M$ is a manifold where we can descend the metric $g$. By abuse of notation we still denote the metric on $M$  as $g$. The metric $g$ defines a canonical flow on the tangent space of $\widetilde{M}$, the geodesic flow. Since  by assumption $\widetilde{M}$ is complete, the geodesic flow is complete as well. The geodesic flow commutes with isometries of $\widetilde{M}$ so it will naturally descend to $TM$. We will denote by $(g_t)_{t\in\R}$ the geodesic flow on $\widetilde{M}$ (the same notation will be used to denote the geodesic flow on $M$). As any autonomous Hamiltonian system, level sets of the  Hamiltonian are invariants by the Hamiltonian flow, in our case the unit tangent bundle $T^1M$ of $M$ is invariant, i.e. the set of vectors with norm one. We also mention that the dynamics on all level sets are conjugate to each other, so there is not loss of generality in restricting to $T^1M$ (this is very particular of our situation).

For us $\mu$ will always stand for a $(g_t)$-invariant probability measure on $T^1M$. Observe that by Poincar\'e recurrence theorem, the measure $\mu$ needs to be supported in the nonwandering set $\Omega$ of the geodesic flow. We now describe $\Omega$ in terms of the action of $\Gamma$ on the boundary at infinity $\partial_\infty\widetilde{M}$ of $\widetilde{M}$. Fix $o\in \widetilde{M}$. Recall that $\partial_\infty\widetilde{M}$ is the set of equivalent classes of asymptotic geodesic rays (two such geodesic rays determine the same class if they are at bounded distance). The unit tangent bundle $T^1\widetilde{M}$ is identified with $(\partial_{\infty}\widetilde{M}\times\partial_{\infty}\widetilde{M})\setminus \text{Diag} \times \R$ via Hopf's coordinates by sending each vector $\tilde{v}\in T^1\widetilde{M}$ to $(v_-,v_+,b_{v_+}(o,\pi(\tilde{v})))$. Here $v_-$ and $v_+$ are respectively the negative and positive ends at infinity of the oriented geodesic line determined by $\tilde{v}$ in $\widetilde{M}$, and $b_{v_+}(o,\pi(\tilde{v}))$ is the Busemann function defined for all $x,y\in\widetilde{M}$ and $\xi\in\partial_\infty\widetilde{M}$ as
$$b_\xi(x,y)=\lim_{t\to\infty} d(x,\xi_t)-d(y,\xi_t),$$
where $t\mapsto \xi_t$ is any geodesic ray ending at $\xi$. Under this identification the geodesic flow acts by translation in the third coordinate. Define the \emph{limit set} $L(\Gamma)$ of $\Gamma$ as the smallest non-empty closed $\Gamma$-invariant subset of $\partial_\infty\widetilde{M}$. Note that for any $z\in \widetilde{M}$, the limit set coincides with $\overline{\Gamma\cdot z}\setminus \Gamma \cdot z$. Finally, the nonwandering set $\Omega$ is identified with $((L(\Gamma)\times L(\Gamma))\setminus \text{Diag} \times \R)/\Gamma$ via Hopf's coordinates.

\subsubsection{Some remarks on entropy and pressure}
\noindent\\
As usual $h_\mu(g)$ is the entropy of the flow with respect to $\mu$, more precisely, the entropy of  $g_1$ with respect to $\mu$.
The notion of topological entropy can be generalized by taking into consideration weights on the points of the phase space. These weights are given by continuous functions $F:T^1M\to \R$ that, for historical reasons, we will call \textit{potentials}. The \textit{topological pressure of a potential} $F$ is defined by $$P(F)=\sup\left\{h_\mu(g)+\int Fd\mu \text{ }\big| \text{ } \mu\text{ a }(g_t) \text{-invariant probability measure}\right\}.$$
We say that a measure $\mu$ is an \textit{equilibrium state} of $F$ if $P(F)=h_\mu(g)+\int Fd\mu$.

\begin{definition} Let $m$ be a $(g_t)$-invariant measure on $T^1M$. We say that $m$ verifies the \emph{Gibbs property for the potential} $F:T^1M\to \R$ if for every compact set $K\subset T^1M$ and $r>0$ there exists a constant $C_{K,r}\geq 1$ such that for every $v\in K$ and every $n\geq 1$ such that $g^n(v)\in K$, we have
$$C_{K,r}^{-1}\leq \frac{m(B_n(v,r))}{e^{\int_0^n F(g_t(v))-P(F)dt}}\leq C_{K,r}.$$
\end{definition}

It turns out that the geodesic flow on $T^1M$ always admits a $(g_t)$-invariant measure verifying the Gibbs property. Moreover, it is the equilibrium measure when finite (see for instance \cite{MR3444431}) and the support of this measure is the entire non-wandering set of the geodesic flow. In the particular case when $F=0$, this measure is called the Bowen-Margulis measure and will be denoted by $m_{BM}$.

From now on, we will always consider $d$ as the Riemannian distance on $T^1M$. For the sake of simplicity we will denote $N_\mu(n,r,\delta)$ the number $N^d_\mu(n,r,\delta)$ defined in the previous subsection.

\begin{lemma}\label{lem:katok_entropy_independence} Let $M=\widetilde{M}/\Gamma$ be a complete Riemannian manifold with pinched negative sectional curvature. Assume that the derivatives of the sectional curvature are uniformly bounded. If $\mu$ is an ergodic $(g_t)$-invariant probability measure on $T^1M$, then for every $0<\delta<1$ and $r>0$, we have
$$h_\mu(g)\leq \liminf_{n\to\infty} \frac{1}{n}\log N_\mu(n,r,\delta).$$
\end{lemma}
\begin{proof}
Let $0<\delta<1$, $0<r\leq r'$ and $n\geq 1$. Fix $\epsilon>0$ small and let $K\subset T^1M$ be a compact set such that $\mu(K)>1-\delta/2$ and  $N_\mu(n,r',\delta/3)\ge N(n,r',K)-\epsilon$. By Birkhoff ergodic theorem there exists a set a compact set $K'$ and $N_0>0$ such that  $\mu(K')>1-\delta/3$,  and $|\frac{1}{n}\int_0^n F(g_t v)-\int Fd\mu|<\epsilon$ for all $v\in K'$ and $n\ge N_0$. Since $\mu(K)>1-\delta/3$, we have $\mu(K\cap g^{-n}K\cap K')>1-\delta$. For $n\ge N_0$ define $K_n=K\cap g^{-n}K\cap K'$. Consider a $(n,r')$-covering with minimal cardinality of $K_n$ and denote by $S$ the set of centers of such dynamical balls.
%By definition $\# S\leq N_\mu(n,r',\delta/2)$.
 For each $v\in S$, let $E_v$ be a $(n,r)$-separated set of maximal cardinality in $B_n(v,r')$. We can moreover assume that $E_v\subset K_n$.  By definition, the $(n,r/2)$-dynamical balls with centers in $E_v$ are disjoint. Moreover, since $\#E_v$ is maximal, the collection of $(n,r)$-dynamical balls having centers in $E_v$ is a $(n,r)$-covering of $B_n(v,r')$. Therefore
\begin{equation*}\label{eq:lem:katok_entropy}
\sum_{w\in E_v}m(B_n(w,r/2))=m\left(\bigcup_{w\in E_v} B_n(w,r/2) \right) \leq m(B_n(v,r')),
\end{equation*}
and so
$$\#E_v \leq \frac{m(B_n(v,r'))}{\min_{w\in E_v}m(B_n(w,r/2))}.$$
Recall now that $m$ satisfies the Gibbs property. As $w,g^n(w)\in K$ by construction, there exists a constant $C\geq 1$, depending only on the compact $K$, $r$ and $r'$, such that $$\#E_v \leq C \exp(\int_0^n F(g_tv)dt- \min_{w\in E_v} \int_0^n F(g_t w)dt).$$ Therefore by the definition of $K'$ we have
$$\#E_v \leq C \exp(2n\epsilon)$$
Observe that
\begin{eqnarray*}
N_\mu(n,r,\delta) &\leq N(n,r,K_n) \leq C\exp(2n\epsilon) \#S = C\exp(2n\epsilon)N(n,r',K_n) \\
&\le C\exp(2n\epsilon)N(n,r',K) \le C\exp(2n\epsilon)(N_\mu(n,r',\delta/3) +\epsilon)
\end{eqnarray*}
Then $$\liminf_{n\to\infty} \frac{1}{n}\log N_\mu(n,r,\delta)\le \liminf_{n\to\infty} \frac{1}{n}\log N_\mu(n,r',\delta/3)+2\epsilon,$$
But $\epsilon>0$ was arbitrary. Finally, using Theorem \ref{thm:katok}, it follows
\begin{eqnarray*}
h_\mu(g)&\leq& \liminf_{n\to\infty} \frac{1}{n}\log N_\mu(n,r,\delta).
\end{eqnarray*}
\end{proof}

In this paper we will be interested in investigate properties of the pressure of H\"older potentials. For this, we start with a few definitions that will be important in following sections.

\begin{definition} Let $F:T^1M\to \R$ be a function and $\widetilde{F}$ its lift to $T^1\widetilde{M}$. Let $G$ be a subgroup of isometries of $\widetilde{M}$. Define the Poincar\'e series associated to $(G,F)$ based at $x\in \widetilde{M}$ as
$$P_G(s,F)=\sum_{\gamma\in G} \exp\left(\int_x^{\gamma x}(\widetilde{F}-s)\right),$$
where $\int_x^{y}L$ denotes the integral of $L$ along the geodesic segment $[x,y]$ (in the universal cover $\widetilde{M}$). The critical exponent of $(G,F)$ is $$\delta^F_G=\inf\{\text{s } | \text{ }P_G(s,F)\text{ is finite}\}.$$
We say $(G,F)$ is of convergence type if $P_G(\delta_G^F,F)$ is finite. Otherwise we say $(G,F)$ is of divergence type.
If $F$ is the zero potential, then we use $\delta_G$ when refering to the critical exponent of $(G,0)$.
\end{definition}
\begin{remark}
Observe that under our assumptions, if $F$ is bounded then $\delta_G^F$ is finite. The critical exponent and being of divergence/convergence type do not depend on the base point $x$ if $F$ is H\"older-continuous (see \cite{coudene}).
\end{remark}

Let $\Gamma$ be a discrete torsion free subgroup of isometries of $\widetilde{M}$. The group $\Gamma$ is said to be elementary if $\#L(\Gamma)<\infty$. Otherwise we say that $\Gamma$ is a non-elementary group. The convergence/divergence type of the group $\Gamma$ has strong ergodic implications for the geodesic flow with respect to the Bowen-Margulis measure as we can see from the theorem stated below (see \cite{Yue}).

\begin{theorem}\label{thm:pts} Let $\Gamma$ be a discrete non-elementary torsion free subgroup of isometries of $\widetilde{M}$. Then, we have
\begin{enumerate}
\item[(a)] the group $\Gamma$ is of divergence type if and only if the geodesic flow $(g_t)$ is ergodic and completely conservative with respect to $m_{BM}$, and
\item[(b)] the group $\Gamma$ is of convergence type if and only if the geodesic flow $(g_t)$ is non-ergodic and completely dissipative with respect to $m_{BM}$, and
\end{enumerate}
\end{theorem}

We recall now the construction of \emph{Gibbs measures} for $(\Gamma,F)$. Let $\delta\in\R$. A \emph{Patterson density} of dimension $\delta$ for $(\Gamma,F)$ is a family of finite nonzero (positive Borel) measures $(\sigma_x)_{x\in\widetilde{M}}$ on $\partial_\infty \widetilde{M}$, such that, for every $\gamma\in\Gamma$, for all $x,y\in\widetilde{M}$, for every $\xi\in\partial_\infty\widetilde{M}$, we have
$$\gamma_\ast\sigma_x=\sigma_{\gamma x} \quad \mbox{and} \quad \frac{d\sigma_x}{d\sigma_y}(\xi)=e^{-C_{F-\delta,\xi}(x,y)},$$
where $C_{F,\xi}(x,y)$ is the \emph{Gibbs cocycle} defined as
$$C_{F,\xi}(x,y)=\lim_{t\to\infty}\int_y^{\xi_t}\widetilde{F}-\int_x^{\xi_t}\widetilde{F},$$
for any geodesic ray $t\mapsto\xi_t$ ending at $\xi$. Note that the limit in the definition of the Gibbs cocycle always exists since the manifold has negative curvature and the potential is H\"older-continuous. It coincides with the Busemann cocycle when $F=-1$.
By a result of Patterson \cite{MR0450547} (see \cite{MR3444431} for the general case), if $\delta^F_\Gamma<\infty$, then there exists at least one Patterson density of dimension $\delta^F_\Gamma$ for $(\Gamma,F)$, with support equal to the limit set $L(\Gamma)$ of $\Gamma$. If $(\Gamma,F)$ is of divergence type then there is only one Patterson density of dimension $\delta^F_\Gamma$. Denote by $(\sigma^{\iota}_x)$ the Patterson density of dimension $\delta^F_\Gamma$ for $(\Gamma,F\circ \iota)$, where $\iota$ is the flip isometry map on $T^1\widetilde{M}$. Using the Hopf parametrisation $v\mapsto (v_-,v_+,t)$ with respect to a base point $o\in\widetilde{M}$, the measure
$$d\tilde{m}(v)=e^{C_{F\circ \iota-\delta^F_\Gamma,v_-}(x_0,\pi(v))+C_{F-\delta^F_\Gamma,v_+}(x_0,\pi(v))}d\sigma^\iota_o(v_-)d\sigma_o(v_+)dt,$$
is independent of $o\in\widetilde{M}$, $\Gamma$-invariant and $(g_t)$-invariant. This induces a measure $m$ on $T^1M$ called the \emph{Gibbs measure associated to the Patterson density} $\sigma_x$.

%Verificar
For the potential $F=0$, the following theorem was proven by Sullivan \cite{MR766265} in the compact case and by Otal-Peign\'e \cite{MR2097356} in the non-compact situation. It was recently extended to arbitrary bounded H\"older potentials by Paulin-Pollicott-Schapira \cite{MR3444431}.

\begin{theorem}\label{thm:pps} Let $\widetilde{M}$ be a complete simply connected Riemannian manifold, with dimension at least 2 and pinched negative sectional curvature. Let $\Gamma$ be a non-elementary discrete group of isometries of $\widetilde{M}$. Let $\widetilde{F}:T^1\widetilde{M}\to\R$ be a bounded H\"older-continuous $\Gamma$-invariant potential and let $F:T^1M\to\R$ be its projection. Then,
\begin{enumerate}
  \item[(a)] the topological pressure satisfies $P(F)=\delta^F_\Gamma$, and
  \item[(b)] if there exists a finite Gibbs measure $m_F$ for $(\Gamma,F)$, then $m^F=m_F/||m_F||$ is the unique equilibrium measure for $(\Gamma,F)$. Otherwise there exists no equilibrium state for $(\Gamma,F)$.
\end{enumerate}
\end{theorem}

\subsubsection{The geometrically finite case}
\noindent\\
Recall that the \emph{Liouville measure} on $T^1M$ is the volume measure on $T^1M$ induced by the Riemannian metric on $M$. The class of manifolds that we will consider in a large part of this paper is defined as follows.

\begin{definition}
A negatively curved Riemannian manifold $M$ is said to be \textit{geometrically finite} if an $\epsilon$-neighborhood of $\Omega$ has finite Liouville measure.
\end{definition}

We recall that Iso$(\widetilde{M})$ has three types of elements: elliptic, hyperbolic and parabolic. The groups we are considering do not contain elliptic elements because of discreteness and the torsion free assumption. We will in general assume the existence of parabolic elements. If $M$ is geometrically finite, then there exist finitely many non-conjugate maximal parabolic subgroups of $\Gamma$. Each of them provide a standard neighborhood of the cusps of $M$, that is a neighborhood isometric to $\mathcal{H}/\mathcal{P}$, where $\mathcal{H}$ is an horoball based at the fixed point of a maximal parabolic subgroup $\mathcal{P}$ of $\Gamma$. It also means that the action of every maximal parabolic group $\mathcal{P}\subset\Gamma$, with fixed point $\xi_\mathcal{P}\in\partial_\infty\widetilde{M}$, is cocompact on $L(\Gamma)\setminus \xi_\mathcal{P}$ (see \cite{MR1317633}). This geometrical fact will be exploted in following sections. In the geometrically finite case, the behaviour of the maximal parabolic subgroups has some implications in the existence of equilibrium measures for the geodesic flow. The following theorem provides a very useful criterion for the existence of an equilibrium measure. This result was first proven in \cite{MR1776078} for the case when $F=0$, then it was extended for bounded H\"older-continuous potentials in \cite{MR3444431} (which is the version presented below).

\begin{theorem} \label{thm:gibbs_finite} Let $M=\widetilde{M}/\Gamma$ be a geometrically finite Riemannian manifold with pinched negative sectional curvature. Assume that the derivatives of the sectional curvature are uniformly bounded. Let $F:T^1M\to \R$ be a bounded H\"older-continuous potential. If $\delta^{F}_{\P}<\delta^{F}_{\Gamma}$ for every parabolic subgroup $\P$ of $\Gamma$, then the Gibbs measure $m_F$ is finite.
\end{theorem}

The next theorem is also proven originally in \cite{MR1776078} for the case $F=0$. It was extended by Coud\`ene \cite{coudene} to a larger family of potentials. Since Coud\`ene's definition of pressure differs with the definition presented in \cite{MR3444431}, for completeness, we give an adaptation of his proof to that context. In both cases the idea follows closely \cite{MR1776078}.
\begin{theorem}[GAP criterion]  \label{thm:gap_criterion} Assume $(\P,F)$ is of divergent type for a parabolic subgroup $\P$. Then $\delta^{F}_{\P}<\delta^{F}_{\Gamma}$.
\end{theorem}
\begin{proof}
Since $\P\subset \Gamma$, we have $\delta^{F}_{\P}\leq\delta^{F}_{\Gamma}$. To prove the strict inequality we will use the divergence of $(\P,F)$. Let $\mathcal{D}_\infty$ be a fundamental domain for the action of $\P$ on $\partial\widetilde{M}$ and let $(\sigma_x)$ be a Patterson density for ($\Gamma,F$) of conformal exponent $\delta^{F}_{\Gamma}$. Since the support of $\sigma_x$ is $L(\Gamma)$ and $L(\Gamma) \neq L(\P)$, we have $\sigma_x(\mathcal{D}_\infty)>0$. Therefore
\begin{eqnarray*}
\sigma_x(L(\Gamma))&\geq& \sum_{p\in\P}\sigma_{x}(p\mathcal{D}_\infty)=\sum_{p\in\P}\sigma_{p^{-1}x}(\mathcal{D}_\infty)\\
&=&\sum_{p\in\P} \int_{\mathcal{D}_\infty}\exp(-C_{F-\delta^{F}_{\Gamma},\xi}(p^{-1}x,x)) d\sigma_x(\xi)\\
&=&\sum_{p\in\P} \int_{\mathcal{D}_\infty}\exp(-\delta^{F}_{\Gamma}B_\xi(p^{-1}x,x)-C_{F,\xi}(p^{-1}x,x)) d\sigma_x(\xi)\\
&\geq& \sum_{p\in\P} \int_{\mathcal{D}_\infty}\exp(-\delta^{F}_{\Gamma}d(x,px)-C_{F,\xi}(p^{-1}x,x)) d\sigma_x(\xi).
\end{eqnarray*}
Note that, for $r>0$ large enough, $\sigma_x$-almost every $\xi\in\mathcal{D}_\infty$ belongs to the shadow $\mathcal{O}_{p^{-1}x}(x,r)$ of $B(x,r)$ seen from $p^{-1}x$. Hence, Lemma 3.4 (2) in \cite{MR1776078} shows that there exists an universal constant $C\geq 1$ such that
\begin{eqnarray*}
C_{F,\xi}(p^{-1}x,x) &\leq& -\int_{p^{-1}x}^x F + C\\
&=& -\int_x^{px}F+C.
\end{eqnarray*}
Therefore
\begin{eqnarray*}
\sigma_x(L(\Gamma)) &\geq& \sum_{p\in\P} \int_{\mathcal{D}_\infty}\exp\left(-\delta^{F}_{\Gamma}d(x,px)+ \int_{x}^{px} F - C\right) d\sigma_x(\xi)\\
&=& e^{-C}\mu_x(\mathcal{D}_\infty)\sum_{p\in\P}\exp\left(\int_{x}^{px} (F-\delta^{F}_{\Gamma})\right).
\end{eqnarray*}
In other words, the series $\sum_{p\in\P}\exp(\int_{x}^{px} (F-\delta^{F}_{\Gamma}))$ is finite. Since $(\P,F)$ is of divergence type we must have $\delta^{F}_{\P}<\delta^{F}_{\Gamma}$.
\end{proof}

\section{Some technical results}\label{tech}
In this section we will always assume $M=\widetilde{M}/\Gamma$ to be a complete geometrically finite Riemannian manifold with pinched negative sectional curvature. We also assume that the derivatives of the sectional curvature are uniformly bounded. For the sake of simplicity, we denote by $X$ the unit tangent bundle $T^1M$ of $M$. Let $\cD$ be a fundamental domain of $\Gamma$ in $\widetilde{M}$ and fix $o\in \cD$. For every $\xi\in\partial_\infty M$ and $s>0$, denote by $B_\xi(s)$ the horoball centered at $\xi$ of height $s$ relative to $o$, that is
$$B_{\xi}(s)=\{y\in \widetilde{M}:b_{\xi}(o,y)\ge s\},$$
where $b_{\xi}(o,\cdot)$ is the Busemann function at $\xi$ relative to $o$. The geometrically finite assumption on the manifold implies that there exists a maximal finite collection $\{\xi_i\}_{i=1}^{N_p}$ of non equivalent parabolic fixed points in $\partial_\infty \cD$, and $s_0>0$, such that for every $s\geq s_0$ the collection $\{B_{\xi_i}(s)/\P_i\}_{i=1}^{N_p}$ are disjoint cusp neighbourhoods for $M$, where $\P_i$ the maximal parabolic subgroup of isometries of $\Gamma$ fixing the point $\xi_i$. The set $\pi(\Omega)\setminus \bigsqcup_{i=1}^{N_p}B_{\xi_i}(s)/\P_i$ is relatively compact in $M$.

For every $s\geq s_0$ define
$$X_{>s}= \bigcup_{i=1}^{N_p} T^1 H_{\xi_i}(s)/\P_i \quad\mbox{and}\quad X_{\le s} = X\setminus X_{>s}.$$
Observe that $\Omega\subset X_{>s} \bigcup X_{\leq s}$.

\begin{lemma}\label{lem:1} Let $s>s_0$. There exists $l_s\in\N$ such that whenever $v\in X_{\le s_0}\cap \Omega$ satisfies $gv,...,g^kv\in X_{\le s}\cap X_{>s_0}$ and $g^{k+1}v\in X_{\le s_0}$, then necessarily we have $k\leq l_s$.
\end{lemma}
\begin{proof}
Let $v\in X_{\le s_0}$ such that $gv,...,g^kv\in X_{\le s}\cap X_{>s_0}$ and $g^{k+1}v\in X_{\le s_0}$. By definition of $X_{>s_0}$, the unit vector $v$ belongs to a standard cusp neighbourhood $T^1 H_{\xi}(s_0)/\mathcal{P}$. Let $\tilde{v}$ be any lift of $v$ into $T^1\widetilde{M}$. Since $g_v\in X_{>s_0}$, we necessarily have $v\in X_{>s_0-1}\cap X_{\leq s_0}$, or equivalently $\tilde{v}\in H_{\xi}(s_0-1)\setminus H_{\xi}(s_0)$. Note that we also have $g^{k+1}\tilde{v}\in H_{\xi}(s_0-1)\setminus H_{\xi}(s_0)$ since $g^{k+1}v\in X_{\le s_0}$. As we are assuming $g^i(v)\in X_{\le s}\cap X_{>s_0}$ for every $1\leq i\leq k$, the geodesic segment $[\pi(\tilde{v}),\pi(g^{k+1}\tilde{v})]$ verifies
$$[\tilde{v},g^{k+1}\tilde{v}]\subset H_{\xi}(s_0-1)\setminus H_{\xi}(s+1).$$
In other words, to find an uniform upper bound of $k$, it suffices to find an uniform upper bound of the length of any geodesic segment inside the set $H_{\xi}(s_0-1)\setminus H_{\xi}(s+1)$ with initial and end points in $\partial H_{\xi}(s_0-1)$. Since we are assuming $v\in\Omega$, we only need to consider those geodesic segments that belong to $\pi(\widetilde{\Omega})$, where $\widetilde{\Omega}$ is the pre-image of $\Omega$ by the natural projection $T^1\widetilde{M}\to T^1M$. Let $D$ be any fundamental domain of the action of $\mathcal{P}$ on $H_{\xi}(s_0-1)$ and denote by $\mathcal{A}$ the set $\mathcal{A}=\partial H_{\xi}(s_0-1)\cap D\cap\pi(\widetilde{\Omega})$. Note that $\mathcal{A}$ is compact since $\mathcal{A}$ is homeomorphic to every fundamental domain of the cocompact action of $\mathcal{P}$ on $L(\Gamma)\setminus\{\xi\}$. Let $\mathcal{P}_{\mathcal{A}}$ be the set of all $p\in\mathcal{P}$ such that for every $x,y\in \mathcal{A}$, we have $[x,py]\subset H_{\xi}(s_0-1)\setminus H_{\xi}(s+1)$. We claim that $\#\mathcal{P}_{\mathcal{A}}<+\infty$. If not, then there exists a sequence $(p_n)$ of parabolic elements in $\P$, and two sequences $(x_n)$ and $(y_n)$ of points in $\mathcal{A}$, such that $[x_n,p_n y_n]\subset H_{\xi}(s_0-1)\setminus H_{\xi}(s+1)$ and $\lim_{n\to\infty} p_n(o) =\xi$, where $o$ is any point in $\widetilde{M}$. But $\mathcal{A}$ is compact, so we can find an increasing sequence of positive integers $(\phi(n))$ such that
\begin{equation*}
\lim_{n\to\infty} x_{\phi(n)}=x' \quad\mbox{and}\quad \lim_{n\to\infty}p_{\phi(n)}y_{\phi(n)}=\xi.
\end{equation*}
Since $[x_{\phi(n)},p_{\phi(n)}y_{\phi(n)}]$ does not intersect $H_{\xi}(s+1)$ for every $n\geq 1$, it follows that the geodesic ray $[x',\xi)$ does not intersect $H_{\xi}(s')$ for every $s'>s+1$, which is a contradiction. Hence, the set $\mathcal{P}_{\mathcal{A}}$ has finite cardinality. Since $\mathcal{A}$ is compact and $\mathcal{P}_{\mathcal{A}}$ is finite, every geodesic segment $[x,py]$, with $x,y\in\mathcal{A}$ and $p\in\mathcal{P}_\mathcal{A}$, has length bounded from above by an uniform constant $l_s(\mathcal{P})>0$. Finally, recall that the geometrically finite assumption on the manifold implies that we have finitely many cusp neighbourhoods intersecting the non-wandering set. In particular, if $l_s>0$ is defined as
$$\sup_\mathcal{P} l_s(\mathcal{P}),$$
the conclusion of this lemma follows.
\end{proof}

We say that an interval $[a,a+b)$ is an excursion of $Y\subset X$ into $X_{>s}$ if
\begin{eqnarray*}
g^{a-1}Y, g^{a+b}Y\subset & X_{\le s_0},\\
g^{a}Y,...,g^{a+b-1}Y\subset &  X_{\ge s_0}
\end{eqnarray*}
Let $n\geq 1$. We denote by $|E_{s,n}(Y)|$ the sum of the length of all the excursions of $Y$ into $X_{>s}$ for times in $[0,n]$. We also denote by $m_{s,n}(Y)$ the number of excursions of $Y$ into $X_{>s}$.

%In the next Proposition we assume $s_0$ and $s_2$ are big enough, the specific properties required will be clear from the proof. (muy vago)
The following lemma despite of being very simple is a fundamental property of geodesic flows on nonpositively curved manifolds. For another application of this property see \cite{MR554385}
\begin{lemma}\label{lem:2} Given $r>0$, there exist $\epsilon(r)>0$ such that if $x,y \in T^1\widetilde{M}$ are such that $d(\pi x,\pi y)<\epsilon(r)$ and $d(\pi g_tx,\pi g_ty)<\epsilon(r)$, then $x$ belong to the $(t,r)$-dynamical ball of $y$ (and viceversa).
\end{lemma}

The following proposition will be used in the proof of Theorem \ref{thm:escape_mass}.

\begin{proposition}\label{prop:partition_Bowen_Balls} Let $r>0$ and $s>s_0$. Suppose that $\beta=\{X_{>s}\cap\Omega,X_{\le s}\cap X_{>s_0}\cap\Omega, Q_1,...,Q_b\}$ is a finite partition of $\Omega$ such that $\diam (g^j(Q_k)) < \epsilon(r)$ for every $0\le j\le l_s$ and every $1\leq k\leq b$. Then there exists a constant $C_{r,s_0}\geq 1$ such that for each $n\geq 1$ and $Q\in \beta^n_0$ with $Q\subset X_{\leq s_0}$, the set $Q$ can be covered by
$$C_{r,s_0}^{m_{s,n}(Q)}e^{\overline{\delta}_\P|E_{s,n}(Q)|}$$
$(n,r)$-dynamical balls.
\end{proposition}
\begin{proof}
Let $Q\in \beta^n_0$ satisfying $Q\subset X_{\leq s_0}$. We decompose $[0,n]$ according to the excursions into $X_{>s_0}$ that contain excursions into $X_{>s}$. To be more precise, we can write
$$[0,n]=W_1\cup \widetilde{V_{1}}\cup ... \cup W_N\cup \widetilde{V_{N}}\cup W_{N+1},$$
where $W_i=[l_i,l_i+L_i)$ and $\widetilde{V_{i}}=[n_i,n_i+h_i)$, with $l_i+L_i=n_i$, $n_i+h_i=l_{i+1}$. Here the decomposition is done so that $\widetilde{V_i}$ represents an excursion into $X_{>s_0}$ that contains an excursion $V_i$ into $X_{>s}$.

The idea of the proof is to make some inductive steps on those intervals ($W_i$ and $\widetilde{V_{i}}$) to control the minimal number of dynamical balls needed to cover $Q$. Denote by $\beta'$ the partition $\beta|_{X_{\leq s_0}}$ or equivalently $\{Q_1,...,Q_b\}$.

We start by covering $X_{\leq s_0}$ with some finite number $C_1=C_1(s_0,\epsilon(r))$ of $(0,\epsilon(r))$-dynamical balls. Since $Q\subset X_{\leq s_0}$ we are also covering this set. We now alternate the following two steps.\\

\noindent
\underline{Step 1:} Assume that we have covered $Q$ by
$$C^{i-1}_{r,s_0}e^{\overline{\delta}_{\P}(|\widetilde{V_{1}}|+...+|\widetilde{V_{i-1}}|)}$$
$(l_i,r)$-dynamical balls. To be more precise, we are assuming that there exists a collection of vectors $\{v_k\}_{k=1}^S$ such that $v_k\in Q$ and
$$Q\subset \bigcup_{k=1}^S B_{l_i}(v_k,r),$$
where $S\leq C^{i-1}_{r,s_0}e^{\overline{\delta}_{\P}(|\widetilde{V_{1}}|+...+|\widetilde{V_{i-1}}|)}$. We claim that
\begin{equation}\label{eq:step1}
Q\subset \bigcup_{k=1}^S B_{l_i+L_i}(v_k,r).
\end{equation}
Since $Q\in \beta^n_0$, we can write $Q=B_0 \cap g^{-1} B_1 \cap ... \cap g^{-n}B_n$, where $B_j\in \beta$ for all $j\in [0,n]$. By definition, we have $g^j(Q)\subset B_j$, so if $B_j\in \beta'$ we get
\begin{equation*}\label{eq:condition_betaprime}
\diam(g^j(Q))\leq \diam(B_j) < \epsilon(r).
\end{equation*}
In particular, if for all $j\in [l_i,l_i+L_i]$ we have $B_j\in \beta'$, then Lemma \ref{lem:2} implies the result. Suppose now that there is an index $j'\in [l_i,l_i+L_i)$ such that $B_{j'}\in\beta'$ and $B_{j'+1}\in \{X_{>s}\cap \Omega,X_{\le s}\cap X_{>s_0}\cap \Omega\}$. In that case Lemma \ref{lem:1} implies that there exists $1< m\leq l_s$ such that $B_{j'+m}\in \beta'$. By construction of $\beta'$, for each $0\leq k \leq m$, we have $\diam(g^k(B_{j'}))<\epsilon(r)$. Hence $\diam(g^j(Q))\leq \epsilon(r)$ for all $j\in [j',j'+m]$. Repeating this argument some finite number of times, we get $\diam(g^j(Q))<\epsilon(r)$ for all $j\in [l_i,l_i+L_i]$. This concludes the first step of the induction process.\\

\noindent
\underline{Step 2:}
After Step 1, we have points $\{v_k\}_{k=1}^S$ in $Q$ such that
$$Q\subset \bigcup_{k=1}^S B_{l_i+L_i}(v_k,r)=\bigcup_{k=1}^S B_{n_i}(v_k,r),$$
where $S\le C^{i-1}_{r,s_0}e^{\overline{\delta}_{\P}(|\widetilde{V_{1}}|+...+|\widetilde{V_{i-1}}|)}$. We proceed to estimate the number of $(n_i+h_i,r)$-dynamical balls necessary to cover $Q$. For this we will cover independently each ball $B_{n_i}(v_k,r)$ by a number smaller than $H_i=H_i(s_0,r)$ of such dynamical balls. Fix any $v\in\{v_k\}$ and denote by $H_i(v)$ the minimal cardinality of a $(n_i+h_i,r)$-covering of $B_{n_i}(v,r)$. Consider the set $X_{\leq s_0+1}\cap X_{>s_0-1}$, which is a finite union of tubular cusp regions. By construction, the vectors $g^{n_i}v$ and $g^{n_i+h_i}v$ belong to some tubular cusp region $T_0 \subset X_{\leq s_0+1} \cap X_{>s_0-1}$. Note that the set $\mathcal{K}_0=T_0\cap \Omega$ is relatively compact in $X$ and it depends only on $s_0$ and $T_0$. Here we are using the geometrically finite assumption. We claim that there exists a constant $C_2=C_2(s_0,r)\geq 1$, such that
\begin{equation}\label{claim:counting}
H_i(v)\leq H_i := C_2e^{\overline{\delta}_\P h_i}.
\end{equation}
Observe that, if inequality \eqref{claim:counting} is satisfied and $C_{r,s_0}$ is defined as $C_{r,s_0}=C_1C_2$, then we can cover $Q$ by
$$S\times H_i \leq C^{i-1}_{r,s_0}e^{\overline{\delta}_{\P}(|\widetilde{V_{1}}|+...+|\widetilde{V_{i-1}}|)} \times C_{r,s_0} e^{\overline{\delta}_\P h_i} = C^{i}_{r,s_0}e^{\overline{\delta}_{\P}(|\widetilde{V_{1}}|+...+|\widetilde{V_{i}}|)}$$
$(n_i+h_i,r)$-dynamical balls, and the second step of the induction process is finished.

Take any finite partition $\xi_0$ of $K_0=\pi(\mathcal{K}_0)$ such that $\diam(\xi_0)\leq \epsilon(r)$. Obviously the partition $\xi_0$ depends on the region $T_0$, but as there are only finitely many of them (one per cuspidal component), we can suppose that $\#\xi_0$ only depends on $s_0$. In order to prove \eqref{claim:counting} we will work on the universal covering of $M$, i.e on $\widetilde{M}$. We do that because the dynamic of the geodesic flow in $T^1\widetilde{M}$ satisfies properties more suited for our purposes. Let $\widetilde{K_0}$ be the lift of $K_0$ into the fundamental domain $\mathcal{D}$. The partition $\xi_0$ induces a finite partition $\tilde{\xi}_0$ of $\tilde{K}_0$ by lifting each element of $\xi_0$ into $\mathcal{D}$. More generally, if $\P$ denotes the maximal parabolic group associated to the cusp corresponding to $T_0$, then any lift $p\tilde{K}_0$ of $K_0$ has the partition $p\tilde{\xi}_0$ associated to it. We will now state Lemma \ref{lem:1} in a more suitable way.

\begin{lemma}\label{lem:angularstone}
For any two unit vectors $\tilde{w}_1,\tilde{w}_2\in \tilde{K}_0$ verifying $\pi(\tilde{w}_2)\in \tilde{\xi}_0(\pi(\tilde{w}_1))$ and $\pi(g^{h}(\tilde{w}_2))\in p\tilde{\xi}_0(\pi(g^{h}(\tilde{w}_1)))$, we have
$$\tilde{w}_2 \in B_h(\tilde{w}_1,r).$$
\end{lemma}
\begin{proof}
By definition $\diam(p\xi_0)\leq \epsilon(r)$ for every $p\in\P$. Hence, if $\pi(\tilde{w}_2)\in \tilde{\xi}_0(\pi(\tilde{w}_1))$ and $\pi(g^{h}(\tilde{w}_2))\in p\tilde{\xi}_0(\pi(g^{h}(\tilde{w}_1)))$, then Lemma \ref{lem:2} applies directly and we get
$$\tilde{w}_2 \in B_h(\tilde{w}_1,r).$$
\end{proof}

Recall that $g^{n_i}v$ and $g^{n_i+h_i}v$ belong to $\mathcal{K}_0$. Hence, the lift $\tilde{v}\in T^1\widetilde{M}$ of $v$ such that $\pi(g^{n_i}\tilde{v})\in \tilde{K}_0$ verifies $\pi(g^{n_i+h_i}\tilde{v})\in p\tilde{K}_0$ for some $p\in\P$. Lemma \ref{lem:angularstone} implies in particular that $H_i(v)$ is bounded from above by
$$(\#\xi_0)^2\#\{p\in\P:\pi(g^{n_i+h_i}\tilde{v})\in p\tilde{K}_0\}.$$
Let $\xi_p\in\partial \widetilde{M}$ be the parabolic fixed point of $\P$ and fix $x_0\in\mathcal{D}$ such that $b_{\xi_p}(o,x_0)=s_0$. We necessarily have the following: if $p\in\P$ is such that $\pi(g^{n_i+h_i}\tilde{v})\in p\tilde{K}_0$, then
$$d(x_0,px_0)-2\Delta \leq h_i \leq d(x_0,p x_0)+2\Delta,$$
where $\Delta=\diam(K_0)$. Again by the finiteness assumption, we can assume $\Delta$ depending only on $s_0$ and not on the chosen tubular cusp region $T_0$. Therefore
$$H_i(v)\leq (\#\xi_0)^2\#\{p\in\P : d(x_0,px_0)-2\Delta \leq h_i \leq d(x_0,p x_0)+2\Delta\}.$$
But, by the definition of critical exponent and the fact that $\delta_\P>0$, there exists a constant $C'_2=C'_2(\delta_\P,x_0)\geq 1$ such that
$$\#\{p\in\P : d(x_0,px_0)-2\Delta \leq h_i \leq d(x_0,p x_0)+2\Delta\} \leq C'_2e^{\delta_\P(h_i+2\Delta)}.$$
Hence
$$H_i(v)\leq (\#\xi_0)^2C'_2e^{2\delta_\P\Delta}e^{\delta_\P h_i}.$$
Define $C_2=(\#\xi_0)^2\max_\mathcal{P}\{C'_2(\P,x_0)\}e^{2\overline{\delta}_\P\Delta}$. This constant depends only on $x_0$, $s_0$ and $r$. Moreover, we have
$$H_i(v)\leq C_2e^{\delta_\P h_i}\leq C_2e^{\overline{\delta}_\P h_i}=H_i,$$
so inequality \eqref{claim:counting} follows.
\end{proof}

\begin{remark}\label{rem:existence_partition_beta} For every $(g_t)$-invariant probability measure $\mu$ on $X$ there exists a partition $\beta$ as in Proposition \ref{prop:partition_Bowen_Balls} such that $\mu(\partial\beta)=0$.
\end{remark}
\begin{remark}\label{rem:bound} Assume $s_0$ big enough. Then for any $(g_t)$-invariant probability measure $\mu$ on $X$ we have $\mu(X_{\le s_0})>0$.
\end{remark}

\begin{proposition}\label{prop:entropy_bounded_from_above} Let $\beta$ be a partition of $\Omega$ as in Proposition \ref{prop:partition_Bowen_Balls}. Then for every $g$-invariant probability measure $\mu$ on $X$, we have
$$h_\mu(g)\leq h_\mu(g,\beta)+\mu(X_{>s_0})\overline{\delta}_\P+\frac{1}{s-s_0}\log C_{r,s_0}.$$
\end{proposition}
\begin{proof}
By the ergodic decomposition theorem, we can assume without loss of generality that $\mu$ is an ergodic measure. To find an upper bound for the measure-theoretic entropy we will use Theorem \ref{thm:katok}, which together Lemma \ref{lem:katok_entropy_independence} implies that for every $0<\delta<1$ and $r>0$, we have
\begin{equation}\label{eq:ent_ineg}
h_\mu(g)\leq \liminf_{n\to\infty} \frac{1}{n}\log N_\mu(n,r,\delta).
\end{equation}
Let $\delta_0=\mu(X_{\leq s_0})$. An important observation is that $\mu(X_{\leq s_0})$ is positive for any invariant probability measure $\mu$ (there are not invariant subsets in $X_{>s_0}$). By ergodicity, we have $\mu(\bigcup_{i=0}^{K} g^{-i}X_{\leq s_0})>1-\frac{\delta_0}{2}$ for all $K\geq 1$ large enough. Observe that for any $m\geq 1$, the intersection of the pre-image of $\bigcup_{i=0}^{K} g^{-i}X_{\leq s_0}$ by $T^m$ with $X_{\leq s_0}$ has measure at least $\delta_0/2$. It follows that there exists an increasing sequence $(n_i)$ (of the form $m_i+j$ for some $j\in[0,K)$), such that
$$\mu(X_{\leq s_0}\cap T^{-n_i} X_{\leq s_0}) > \frac{\delta_0}{2K}.$$
By Shannon-McMillan-Breiman Theorem, for all $\varepsilon>0$ the set
$$A_{\varepsilon,N}:=\{x\in X : \forall n\geq N, \mu(\beta^n_0(x))\geq \exp(-n(h_\mu(g,\beta)+\varepsilon))\}$$
has measure converging to 1 as $N$ goes to $\infty$. Hence, given any $0<\delta_1<1$ the set $A_{\delta_1,N}$ has measure $>1-\delta_1$ for $N$ large enough (depending on $\delta_1$). Observe that $A_{\delta_1,n}$ can be covered by at most $\exp(n(h_\mu(g,\beta)+\delta_1))$ many elements in $\beta^n_0$. By taking $\delta_1 =\frac{\delta_0}{4K}$ we get $\mu(A_{\delta_1,n_i}\cap Y_i)> \frac{\delta_0}{4K}$, where $Y_i=X_{\leq s_0}\cap T^{-n_i} X_{\leq s_0}$. Birkhoff ergodic theorem tell us that there exists a set $W_{\delta_0}$ of measure $\mu(W_{\delta_0})>1-\frac{\delta_0}{8K}$, such that
$$\frac{1}{n}\sum_{i=0}^{n} \mathbbm{1}_{X_{\geq s_0}}(g^n x) < \mu(X_{\geq s_0})+\frac{\delta_0}{8K},$$
for all $n$ large enough. We finally define
$$X_i= W_{\delta_0}\cap Y_i \cap A_{\delta_1,n_i}.$$
Our goal is to cover $X_i$ by $(n_i,r)$-dynamical balls. Notice that $\mu(X_i)>\frac{\delta_0}{8K}$ by construction. Moreover $X_i$ can be covered by $\exp(n_i(h_\mu(g,\beta)+\delta_1))$ many elements of $\beta^{n_i}_0$. We will use Proposition 2 to control the number of dynamical balls needed to cover each element of this partition. Let $Q\in \beta^{n_i}_0$. By choice of $W_{\delta_0}$, we have $|E_{s,n}(Q)|<(\mu(X_{\geq s_0})+\frac{\delta_0}{8K})n_i$. Note also that any orbit needs at least time $s-s_0$ to go from $X_{\leq s_0}$ to $X_{>s}$, hence $m_{s,n_i}(Q)\leq n_i/(s-s_0)$. Putting all together, we get that $N(n_i,r,1-\delta_0/8K)$ is bounded from above by
$$\exp(n_i(h_\mu(g,\beta)+\delta_1)) C_{r,s_0}^{n_i\frac{1}{(s-s_0)}} e^{n_i\overline{\delta}_\P(\mu(X_{\geq s_0})+\frac{\delta_0}{8K})}.$$
Using inequality \eqref{eq:ent_ineg}, we obtain
\begin{eqnarray*}
h_\mu(g)&\leq& \liminf_{n\to\infty} \frac{1}{n}\log N(n,r,1-\delta_0/8K)\\
&\leq& \limsup_{i\to\infty}\frac{1}{n_i}\log N(n_i,r,1-\delta_0/8K)\\
&\leq& h_\mu(g,\beta)+\frac{\delta_0}{4K} +\overline{\delta}_\P\left(\mu(X_{\geq s_0})+\frac{\delta_0}{8K}\right)+\frac{1}{s-s_0}\log C_{r,s_0}.
\end{eqnarray*}
Finally, it suffices to take $K\to\infty$ to conclude.
\end{proof}

\section{Escape of mass} \label{escape}
In this section using results obtained in section \ref{tech}  we conclude the uppersemicontinuity result about measure theoretic entropry stated before. \\\\
\noindent
\textbf{Theorem \ref{thm:escape_mass}} \emph{Let $(M,g)$ be a geometrically finite Riemannian manifold with pinched negative sectional curvature and of divergence type. Assume that the derivatives of the sectional curvature are uniformly bounded. If $(\mu_n)$ is a sequence of $(g_t)$-invariant probability measures on $T^1M$ such that $\mu_n \rightharpoonup \mu$, then}
$$\limsup_{n\to\infty} h_{\mu_n}(g) \leq \|\mu\|h_{\frac{\mu}{\|\mu\|}}(g)+(1-\|\mu\|)\overline{\delta}_{\P}.$$
\begin{proof}
Pick $s>s_0$ such that $\mu(\partial X_{\leq s})=0$. Let $\beta$ be a partition of $X$ as in Proposition \ref{prop:partition_Bowen_Balls} such that $\mu(\partial \beta)=0$. Let $\varepsilon>0$ and assume $\mu(X)>0$. By definition of the measure-theoretic entropy, we fix $m\in\N$ such that
$$h_{\frac{\mu}{\|\mu\|}}(g)+\varepsilon>\frac{1}{m}H_{\frac{\mu}{\|\mu\|}}(\beta^{m-1}_0), \quad 2\frac{e^{-1}}{m}<\frac{\varepsilon}{2},$$
and $-(1/m)\log \mu(X)<\varepsilon$. Using the definition of entropy of a partition, and the estimations above, we get
$$\mu(X)h_{\frac{\mu}{\|\mu\|}}(g)+2\varepsilon > -\frac{1}{m}\sup_{Q\in \beta^{m-1}_0}\mu(Q)\log\mu(Q).$$
Define $A=\bigcap_{i=0}^{m-1}g^{-i} X_{>s}$. Observe that
$$\lim_{n\to\infty} \sum_{Q\in \beta^{m-1}_0\setminus\{A\}}\mu_n(Q)\log\mu_n(Q) = \sum_{Q\in \beta^{m-1}_0\setminus\{A\}}\mu(Q)\log\mu(Q),$$
so we can find $n_0\in\N$ such that for all $n\geq n_0$ the term
$$\left|-\frac{1}{m}\sum_{Q\in \beta^{m-1}_0}\mu(Q)\log\mu(Q)-\frac{1}{m}H_{\mu_n}(\beta^{m-1}_0) \right|$$
is bounded from above by
\begin{eqnarray*}
\frac{1}{m}\left| \sum_{Q\in \beta^{m-1}_0\setminus\{A\}} (\mu_n(Q)\log\mu_n(Q)-\mu(Q)\log\mu(Q)) \right|&+&\frac{1}{m}|\mu_n(A)\log\mu_n(A)-\mu(A)\log\mu(A)|\\
&\leq& \frac{\varepsilon}{2}+2\frac{e^{-1}}{m}<\varepsilon.
\end{eqnarray*}
We use now Proposition \ref{prop:entropy_bounded_from_above} to get
$$h_{\mu_n}(g)\leq h_{\mu_n}(g,\beta)+\mu_n(X_{>s_0})\overline{\delta}_\P+\frac{1}{s-s_0}\log C_{r,s_0}.$$
In particular,
\begin{eqnarray*}
\mu(X)h_{\frac{\mu}{\|\mu\|}}(g)+3\varepsilon &>& \frac{1}{m}H_{\mu_n}(\beta^{m-1}_0)\geq h_{\mu_n}(g,\beta)\\
&>& h_{\mu_n}(g)-\overline{\delta}_\P(1-\mu_n(X_{\leq s_0}))-\frac{1}{s-s_0}\log C_{r,s_0}.
\end{eqnarray*}
Hence
$$\limsup_{n\to\infty} h_{\mu_n}(g)\leq \mu(X)h_{\frac{\mu}{\|\mu\|}}(g)+\overline{\delta}_\P(1-\mu(X_{\leq s}))+\frac{1}{s-s_0}\log C_{r,s_0}+3\varepsilon.$$
When $s\to\infty$ and $\varepsilon\to 0$, we finally obtain
$$\limsup_{n\to\infty} h_{\mu_n}(g)\leq \mu(X)h_{\frac{\mu}{\|\mu\|}}(g)+\overline{\delta}_\P(1-\mu(X)).$$

The case when $\mu(X)=0$ follows directly from Proposition \eqref{prop:entropy_bounded_from_above} since $h_{\mu_n}(g,\beta)\to 0$ and $\mu_n(X_{>s_0})\to 1$ as $n$ tends to $\infty$.
\end{proof}

\begin{corollary}\label{cor:leq_delta_P} Let $M$ be a complete geometrically finite Riemannian manifold with pinched negative sectional curvature. Assume that the derivatives of the sectional curvature are uniformly bounded. If $(\mu_n)$ is a sequence of $(g_t)$-invariant probability measures on $T^1M$ such that $\mu_n \rightharpoonup 0$, then
$$\limsup_{n\to\infty} h_{\mu_n}(g) \leq \overline{\delta}_{\P}.$$
\end{corollary}

\section{Pressure map}\label{appli}

In this section we will investigate properties of the pressure of certain potentials using Theorem \ref{thm:escape_mass}. The potentials of interest will be going to zero through the cusps, to be more precise we have the following definition.

\begin{definition} \label{def-pot}We say that a continuous potential $F:T^1M\to \R^+$ \emph{converge to zero through the cusps} if there exists a sequence $(K_n)$ of compact subsets of $M$ such that
\begin{enumerate}
\item[(1)] for every $n\geq 1$, we have $K_n\subset K_{n+1}$,
\item[(2)] the non-wandering set $\Omega$ of the geodesic flow satisfies $\Omega\subset \bigcup_{n\geq 1} T^1 K_n$,
\item[(3)] there exists a sequence $(\epsilon_n)\searrow 0$ of positive numbers such that for every $v\in \Omega\setminus T^1 K_n$, we have
$$|F(v)|\leq \epsilon_n.$$
\end{enumerate}
We define the family $\mathcal{F}$ to be the set of H\"older continuous, positive potentials converging to zero through the cusps.
\end{definition}

\begin{remark} Every positive continuous potential converging to zero through the cusps is bounded on a neighbourhood of the non-wandering set.
\end{remark}

The following two easy lemma will be used in the proof of Theorem \ref{thm:phase_transitions}.

\begin{lemma}\label{lem:convergence_to_zero} Let $F:X\to\R$ be a positive continuous function. If $(\mu_n)$ is a sequence of $(g_t)$-invariant probability measures on $T^1M$ such that
$$\lim_{n\to\infty}\int F d\mu_n = 0.$$
Then $\mu_n \rightharpoonup 0$.
\end{lemma}
\begin{proof}
Let $K\subset X$ be a compact set. Then
\begin{eqnarray*}
\int F d\mu_n &\geq& \int_K Fd\mu_n\\
&\geq& \min\{F(v):v\in K\}\mu_n(K).
\end{eqnarray*}
Observe that $\min\{F(v):v\in K\}>0$ since $F$ is positive, so
$$\lim_{n\to\infty}\int F d\mu_n = 0 \Rightarrow \lim_{n\to\infty}\mu_n(K)=0,$$
which is equivalent to $\mu_n \rightharpoonup 0$.
\end{proof}

\begin{lemma}\label{lem:wconvergence_F} Let $F\in\mathcal{F}$. If $(\mu_n)$ is a sequence of probability measures on $T^1M$ converging vaguely to $\mu$. Then
$$\lim_{n\to\infty}\int Fd\mu_n = \int F d\mu.$$
\end{lemma}
\begin{proof}
Let $\varepsilon>0$ and $K\subset T^1M$ a compact set such that $F|_{T^1M\setminus K} \leq \varepsilon$. As
$$\int_K F d\mu_n \leq \int F d\mu_n \leq \int_K F d\mu_n + \varepsilon$$
and
$$\lim_{n\to\infty} \int_K F d\mu_n = \int_K F d\mu,$$
we have
\begin{equation}\label{eq:lemwconvergence}
\int_K F d\mu \leq \liminf_{n\to\infty} \int F d\mu_n \leq \limsup_{n\to\infty} \int F d\mu_n \leq \int_K F d\mu+\varepsilon.
\end{equation}
Observe that \eqref{eq:lemwconvergence} is verified for every compact $K$ large enough, therefore
$$\int F d\mu \leq \liminf_{n\to\infty} \int F d\mu_n \leq \limsup_{n\to\infty} \int F d\mu_n \leq \int F d\mu + \varepsilon.$$
Since $\varepsilon>0$ is arbitrary, the conclusion of this lemma follows.
\end{proof}

For a $(g_t)$-invariant probability measure we define the \emph{pressure of $F$ with respect to} $\mu$ as $$P(F,\mu)=h_\mu(g)+\int Fd\mu.$$
Observe that Lemma \ref{lem:wconvergence_F} and Theorem \ref{thm:escape_mass} implies the following.

\begin{theorem}\label{existence_potentials} Under the hypothesis of Theorem \ref{thm:escape_mass} and for $F\in\mathcal{F}$ we have
$$\limsup_{n\to\infty} P(F,\mu_n) \leq \|\mu\|P(F,\mu/|\mu\|)+(1-\|\mu\|)\overline{\delta}_{\P}.$$
\end{theorem}
\begin{corollary} \label{cor}Assume the hypothesis of Theorem \ref{thm:escape_mass}. Let $F\in\mathcal{F}$ and assume that $P(F)>\overline{\delta}_\P$. Then the potential $F$ has a unique equilibrium measure.
\end{corollary}

This is an improvement of Theorem \ref{thm:gibbs_finite} in the case $F\in \F$. Corollary \ref{cor} is very convenient when proving existence of equilibrium measures since the gap required does not involve $\delta^F_\P$ which is difficult to handle (and in general bigger than $\delta_\P$). We also mention that for $F$ a bounded negative potential the same inequality holds, but that case is completely covered by Theorem \ref{thm:gibbs_finite}.

\begin{theorem}\label{thm:phase_transitions} Let $M=\widetilde{M}/\Gamma$ be a geometrically finite Riemannian manifold with pinched negative sectional curvature. Assume that the derivatives of the sectional curvature are uniformly bounded. Then every potential $F\in\mathcal{F}$ verifies
\begin{enumerate}
\item[(1)] for every $t \in \R$ we have that $P(t F) \geq \overline{\delta}_\P$
\item[(2)] the function $t\mapsto P(tF)$ has a horizontal asymptote at $-\infty$, that is
$$ \lim_{t \to -\infty} P(tF)= \overline{\delta}_\P.$$
\end{enumerate}
Moreover, if $t':= \inf \left\{ t \leq 0 : P(tF)= \overline{\delta}_\P \right\}$, then
\begin{enumerate}
\item[(3)] for every $t>t'$ the potential $tF$ has a unique equilibrium, and
\item[(4)] the pressure function $t\mapsto P(tF)$ is differentiable in $(t',\infty)$, and it verifies
\begin{equation*}
P(tF)=
\begin{cases}
\overline{\delta}_\P & \text{ if } t < t'\\
\text{strictly increasing}  & \text{ if } t > t',
\end{cases}
\end{equation*}
\item[(5)] If $t<t'$ then the potential $tF$ has not equilibrium measure.
\end{enumerate}
\end{theorem}
\begin{proof}

Let $\P$ be a maximal parabolic subgroup of $\Gamma$ such that $\delta_\P=\overline{\delta}_\P$. Note that since $\P\subset\Gamma$, then $\delta^{tF}_{\P} \leq \delta^{tF}_{\Gamma}$. For $0\leq t$ we have $\delta_{\P}\leq \delta^{tF}_{\P}$, in particular $\overline{\delta}_\P \leq P(tF)=\delta^{tF}_\Gamma$. We are going to prove now that for every $t\leq 0$, we have
$$\delta^{tF}_{\P}=\delta_\P.$$
Fix $t\geq 0$. Again, since $F$ is positive, we get $\delta^{-tF}_{\P}\leq \delta_\P$. For $\varepsilon>0$, there exists a compact set $K\subset M$ such that for every $v\in \Omega \setminus T^1K$, we have $F(v)<\varepsilon$. We can suppose without loss of generality that $M\setminus K$ contains a standard cusp neighborhood $\mathcal{C}$ associated to the parabolic group $\P$. Pick $x_0\in\partial \mathcal{C}$. Then, for every $p\in \P$, we have
\begin{equation*}
\int_{x_0}^{p x_0} F \leq \varepsilon d(x_0,px_0).
\end{equation*}
Therefore, the Poincar\'e series associated to $(-tF,\P)$ is related to the Poincar\'e series of $\P$ as follows:
\begin{equation*}
\sum_{p\in\P}\exp\left( \int_{x_0}^{px_0} (-tF - s) \right) \geq  \sum_{p\in\P} \exp(-(t\varepsilon+s)d(x_0,px_0)).
\end{equation*}
Hence, if $t\varepsilon + s < \delta_\P$, then $\sum_{p\in\P}\exp\left( \int_{x_0}^{px_0} (-tF - s) \right)$ diverges. This implies that $\delta_\P-t\varepsilon \leq \delta_\P^{-tF}$ for every $\varepsilon>0$. In particular, the parabolic critical exponent verifies $\delta_\P \leq \delta^{-tF}_\P$, which concludes the proof of $(1)$.

In order to prove (2) recall that the pressure function $t\mapsto P(tF)$ is convex. If $F\in\mathcal{F}$, from (1) we deduce that this function has an horizontal asymptote as $t\to -\infty$, so the limit $\lim_{t\to -\infty} P(tF)=: A$ exists. Again from (1) we get $A\geq \overline{\delta}_\P$. By variational principle, for every $n\geq 1$ there exists a $(g_t)$-invariant probability measure $\mu_n$ such that $\lim_{n\to\infty} h_{\mu_n}(g)-n\int F d\mu_n=A$. Since $A$ is bounded below we necessarily have
$$\lim_{n\to\infty} \int F d\mu_n = 0.$$
But $F$ goes to zero through the cusps, then by Lemma \ref{lem:convergence_to_zero}, we have $\mu_n \rightharpoonup 0$. Using Corollary \ref{cor:leq_delta_P} and the fact that $A\leq \limsup_{n\to\infty} h_{\mu_n}(g)$, we obtain $A\leq \overline{\delta}_\P$. This implies that $A=\overline{\delta}_\P$ and the desired conclusion follows.

\begin{proposition}\label{prop:diff_equilibrium} Let $F\in\mathcal{F}$ and $t_0\in \R$. Assume that there exists an open interval $I\subset \R$ containing $t_0$ such that for every $t\in I$ the potential $tF$ verifies $P(tF)>\overline{\delta}_\P$. Then the pressure function $t\mapsto P(tF)$ is differentiable at $t_0$ and
$$\frac{d}{dt|}_{t=t_0} P(tF)=\int F d\mu_{t_0}.$$
\end{proposition}
\begin{proof}
The strategy of this proof is analogous to that one in [Keller] where the author treats the compact case (see also [Barreira]). However, in our setting we must be careful because of the escape of mass phenomenon.

Recall that the gap assumption $P(tF)>\overline{\delta}_\P$ implies that every potential $tF$, where $t\in I$ (by Corollary \ref{cor}) admits a unique equilibrium measure. Using the variational principle, for every $t>t_0$ such that $t\in I$, we have
\begin{equation}\label{eq:prop_1}
\int F d\mu_{t_0} \leq \frac{P(tF)-P(t_0F)}{t-t_0}\leq \int F d\mu_t,
\end{equation}
and for every $t<t_0$ such that $t\in I$, we have
\begin{equation}\label{eq:prop_2}
\int F d\mu_{t_0} \geq \frac{P(tF)-P(t_0F)}{t-t_0}\geq \int F d\mu_t.
\end{equation}

Now take any sequence $(t_n)$ of real numbers in $I$ converging to $t_0$. By passing to a subsequence we can assume that $(\mu_{t_n})$ converge vaguely to $\mu$. By Theorem \ref{thm:escape_mass} the mass of $\mu$ is at least $\varepsilon/(\delta_\Gamma-\overline{\delta}_\P)$, so the normalized measure $\overline{\mu}$ of $\mu$ is well-defined. We claim that $\overline{\mu}=\mu_{t_0}$. It suffices to prove that $\overline{\mu}$ is a equilibrium measure for $t_0F$. On the one hand, the variational principle implies the inequality
$$P(t_0F)\geq h_{\overline{\mu}}(g)+t_0\int F d\overline{\mu}.$$
By using Theorem \ref{thm:escape_mass} and Lemma \ref{lem:wconvergence_F} we have
\begin{eqnarray*}
\lim_{n\to\infty} h_{\mu_{t_n}}(g)+t_n\int F d\mu_n &\leq& \|\mu\|h_{\overline{\mu}}(g)+(1-\|\mu\|)\overline{\delta}_\P + t_0\int F d\mu\\
&=& \|\mu\|\left(h_{\overline{\mu}}(g)+ t_0\int F d\overline{\mu}\right)+(1-\|\mu\|)\overline{\delta}_\P.
\end{eqnarray*}
Because of the continuity of the pressure we observe $P(t_n F)>\overline{\delta}_\P$. Then
$$\overline{\delta}_\P \leq h_{\overline{\mu}}(g)+t_0\int Fd\overline{\mu}.$$
and therefore
$$P(t_0F)=\lim_{n\to\infty}P(t_nF)\leq h_{\overline{\mu}}(g)+ t_0\int F d\overline{\mu}.$$
In particular, the measure $\overline{\mu}$ is an equilibrium measure for $t_0F$ so it coincides with $\mu_{t_0}$. Using inequalities \eqref{eq:prop_1} and \eqref{eq:prop_2}, together Lemma \ref{lem:wconvergence_F}, we conclude that $t\mapsto P(tF)$ is differentiable at $t=t_0$, with
$$\lim_{t\to t_0} \frac{P(tF)-P(t_0F)}{t-t_0}=\int F d\mu_{t_0}.$$
\end{proof}

Claim (3) is a direct consequence of Proposition \ref{prop:diff_equilibrium}. By definition of $t'$, the pressure function $t\mapsto P(tF)$ is constant on $(-\infty,t')$. Since the monotony of a differentiable function can be described in terms of the sign of its derivative, and $\int F d\mu>0$ for every $(g_t)$-invariant probability measure $\mu$, the formula of the derivative of $t\mapsto P(tF)$ on $(t',\infty)$ in Proposition \ref{prop:diff_equilibrium} implies that $t\mapsto P(tF)$ is strictly increasing on that range, this concludes the proof of (4).

To prove (5) assume $\mu$ is a equilibrium measure for $tF$, where $t<t'$. If $t<t''<t'$, then $P(tF)=P(tF,\mu)<P(t''F,\mu)\le P(t''F)$, but $P(t''F)=P(tF)$ which is a contradiction.

\end{proof}

We now present a construction of a potential $F$ so that $tF$ has a equilibrium measure for all $t\in \R$. This will be used later on to provide a family of Gibbs measures convering to zero but with high entropy.\\

\noindent
\textbf{Example.} (\emph{No phase transitions}) We now construct a potential $F\in\mathcal{F}$ such that $\overline{\delta}_\P<P(tF)$ for every $t\in \R$. By Corollary \ref{cor} this implies that there are no phase transitions. The construction of this potential follows ideas in \cite{coudene}.

Let $\P$ be a maximal parabolic subgroup of $\Gamma$ such that $\overline{\delta}_\P=\delta_{\P}$. Assume that $\P$ is of divergence type. Recall that $M$ can be decomposed into a compact part and a finite number of cusps regions of the form $\mathcal{H}_i/\P_i$, where $\mathcal{H}_i=\{B_{\xi_i}(s_0)\}_{i=1}^{N_p}$ are disjoint horoballs in $\widetilde{M}$ and each $\P_i$ is a maximal parabolic group fixing $\xi_i$. We are going to construct a potential $F$ on the horoball $\mathcal{H}(s_0)$ corresponding to $\P$, which fixes $\xi$ at the boundary of $\widetilde{M}$. The definition of $F$ in the remaining horoballs can be done analogously.

We will choose two sequences of numbers: an increasing sequence $\{k_n\}_{n\in\N}$ of natural numbers and a decreasing sequence $\{\epsilon_n\}_{n\in\N}$ positive numbers converging to zero. We are going to choose these sequences in the following order: we first choose $\epsilon_1$, then $M_1$, then $\epsilon_2$, then $M_2$, etc. We pick once and for all $x_0 \in \partial \mathcal{H}(s_0)/\P$. For every $l\in\R$ and $L\in\R\cup\{\infty\}$ such that $l<L$, we define
$$S(l,L)=\{p\in\P: l<d(x_0,px_0)<L\}.$$
Define $P^L_l(\epsilon)=\sum_{p\in S(l,L)} \exp(-(\epsilon+ \delta_\P) d(x_0,px_0))$. By definition of the critical exponent of $\P$ this function is finite if $\varepsilon>0$ and since $\P$ is of divergence type we have $\lim_{\epsilon\to 0^+} P_l^\infty (\epsilon)=\infty$, for any $l\geq 1$. Clearly $P_l^L(\epsilon)$ is strictly decreasing in $\epsilon$.\\
Define $\epsilon_1$ so that $P_0^\infty(\epsilon_1)>2$, we assume $\epsilon_1<1$. We choose $k_1$ such that
$$P^{k_1}_{0}(\epsilon_1)=\sum_{p\in S(0,k_1)}\exp(-(\epsilon_1+ \delta_P) d(x_0,px_0))>1.$$
{\bf Claim} There exist $k_1'>k_1$ such that if $p\in S(k_1',\infty)$ then
$$\text{length}([x_0,px_0]\cap (\widetilde{M}\setminus \mathcal{H}(s_0+1)))< \dfrac{1}{2}\epsilon_1 d(x_0,px_0),$$
where $[x,y]$ is the geodesic connecting $x$ and $y$ in the universal covering $\widetilde{M}$. We choose now $\epsilon_2$ satisfying $0<\epsilon_2<\epsilon_1$ and $P^\infty_{k_1'}(2\epsilon_2)>2$. Then we pick $k_2$ large enough so that $k_2>k_1'>k_1$ and $P^{k_2}_{k_1'}(2\epsilon_2)>1$. The construction follows by induction. Assume we have already defined $\epsilon_1,k_1,...,\epsilon_m,k_m$. We want to define $\epsilon_{m+1}$ and $k_{m+1}$. As before, there exist $k_m'>k_m$ such that if $p\in S(k_m',\infty)$ then
$$\text{length}([x_0,px_0]\cap M_{<m})< \dfrac{1}{2}\epsilon_{m} d(x,px),$$
where $M_{<m}=\widetilde{M}\setminus \mathcal{H}(s_0+m)$. Define $\epsilon_{m+1}$ so that $0<\epsilon_{m+1}<\epsilon_m$ and $P^\infty_{k_m'}( (m+1)\epsilon_{m+1})>2$. We finally define $k_{m+1}$ so that $P^{k_{m+1}}_{k_m'}((m+1)\epsilon_{m+1})>1$.\\
We now construct the function $f$ in the horoball $\mathcal{H}(s_0)$. We start by defining the function $d: \mathcal{H}(s_0)/\P \to \R$ by the relation $d(x)=s$ if and only if $\pi(x)\in \partial H(s)/\P$. Note that $d$ is a Lipschitz function and the level sets of $d$ are horospheres centered at $\xi$. We define $f:\mathcal{H}(s_0)/\P\to \R$ as
$$\begin{array}{ll}
  f(x)=\frac{\varepsilon_n}{2}-d(x)+n & \text{ on } d^{-1}([n,n+\varepsilon_n-\varepsilon_{n+1}])  \\
  f(x)=\frac{\varepsilon_{n+1}}{2} & \text{ on }   d^{-1}([n+\varepsilon_n-\varepsilon_{n+1},n+1]).
\end{array}$$
As said before, since the horoballs are disjoint we can define $f$ in every cusp $\mathcal{H}_i/\P_i$ in a similar way. Since $f$ is constant in the boundary of each of those horoballs, we can extend it to the whole manifold into a bounded H\"older-continuous function that goes to zero through the cusps. Now we define $F:T^1M\to \R$ as $F(v)=f(x)$ whenever $\pi(v)=x$. Clearly $F$ is positive bounded and H\"older-continuous by construction. For given $t\geq 0$, choose a natural number $N$ greater than $t$, then
\begin{eqnarray*}
\sum_{p\in\P}\exp\left(\int_{x_0}^{px_0} (-tF - \delta_\P)\right)&=& \sum_{n=1}^\infty \sum_{S(k_n,k_{n+1})} \exp\left(\int_{x_0}^{px_0} (-tF - \delta_\P)\right)  \\
&\ge& \sum_{n=1}^\infty \sum_{S(k_n',k_{n+1})} \exp(-(t\epsilon_n + \delta_\P) d(x_0,px_0))\\
&\ge& \sum_{n=1}^\infty\sum_{S(k_n',k_{n+1})} \exp(-(N \epsilon_n + \delta_\P) d(x_0,px_0))\\
&\ge& \sum_{n=N}^\infty\sum_{S(k_n',k_{n+1})} \exp(-(n \epsilon_n + \delta_\P) d(x_0,px_0))\\
&=& \sum_{n=N}^\infty P_{k_n'}^{k_{n+1}}( n\epsilon_n)>\sum_{n=N}^\infty P_{k_n'}^{k_{n+1}}((n+1)\epsilon_{n+1}),
\end{eqnarray*}
which clearly diverges by construction. In other words $(\P,F)$ is of divergent type. Finally, we get $\overline{\delta}_\P<P(tF)$ for every $t\leq 0$ from the GAP criterion. The case when $t>0$ follows from the fact that $\overline{\delta}_\P<\delta_{\Gamma}$ and that for every $t>0$ we have $\delta_\Gamma \leq P(tF)$.\\

\noindent

\begin{remark}
Since the numbers $\{\epsilon_n\}$ in the above construction are going to zero quite fast, we interprete the decay of $F$ through the cusp associated to $\P$ as `very fast'. On the contrary, if we want to find phase transitions we need to consider potentials with very slow decay in very particular manifolds. This will be treated in a different paper by the second author (see \cite{Vel}).
\end{remark}

We are now in position to prove Theorem \ref{thm:existence_sequence} in the introduction.\\\\
\noindent
\textbf{Theorem \ref{thm:existence_sequence}} \emph{Let $(M,g)$ be a geometrically finite Riemannian manifold with pinched negative sectional curvature. Then there exists a sequence $(\mu_n)$ of ergodic $(g_t)$-invariant probability measures on $T^1M$ converging vaguely to 0 and}
$$\lim_{n\to\infty} h_{\mu_n}(g) = \overline{\delta}_{\P}.$$
\emph{If we moreover assume that there exist a parabolic subgroup $\P$ of divergence type for which $\delta_\P=\overline{\delta}_\P$, then we can take the measures $(\mu_n)$ as Gibbs measures.}

\begin{proof}
Fix a potential $F$ in $\mathcal{F}$. By (2) in Theorem \ref{thm:phase_transitions} and the variational principle, for every $\varepsilon>0$ there exists $N\in\N$ such that for every $n\geq N$ there exists a $(g^t)$-invariant probability measure $\mu_n$ in $T^1M$ such that
\begin{equation}\label{eq:1:construction1}
\overline{\delta}_\P - \varepsilon \leq h_{\mu_n}(g)-n\int F d\mu_n \leq \overline{\delta}_\P+\varepsilon.
\end{equation}
Observe in particular that $\overline{\delta}_\P - \varepsilon \leq h_{\mu_n}(g)$. On the other hand, since the entropy of the geodesic flow is bounded, the central term in \eqref{eq:1:construction1} must to remain bounded. Hence
$$\lim_{n\to\infty}\int F d\mu_n = 0.$$
Note the the measures $\mu_n$ are supported in $\Omega$, so $\mu_n \rightharpoonup 0$ since $F$ goes to zero through the cusps. Therefore, Corollary \ref{cor:leq_delta_P} implies that
$$\limsup_{\mu_n}h_{\mu_n}(g)\leq \overline{\delta}_\P.$$
Putting all together, we get
$$\overline{\delta}_\P - \varepsilon \leq \limsup_{\mu_n}h_{\mu_n}(g)\leq \overline{\delta}_\P,$$
which obviously implies the desired conclusion as $\varepsilon\to 0$. \\
Under the additional hypothesis that there exist a parabolic subgroup $\P$ of divergence type for which $\delta_\P=\overline{\delta}_\P$ we can use the no phase transition example described before as our $F$. In this case we do not need to do any kind of approximation, we just take the respective equilibrium measures since they exist. The argument follows exactly as in the previous case.
\end{proof}

\section{Entropy at infinity for normal coverings}\label{infinity}

A Riemannian manifold $(M,g)$ is called a \emph{regular} $\Z$-\emph{cover} of  $(M_0,g_0)$ if there is a surjective map $p:M\to M_0$ such that
\begin{enumerate}
  \item[(1)] every $x\in M_0$ has a neighbourhood $V_x$ such that every connected component of $p^{-1}(V_x)$ is mapped isometrically by $p$ onto $V_x$,
  \item[(2)] the group $\text{Deck}(M,p)$ of \text{deck transformations}, that is the group
  $$\text{Deck}(M,p):=\{D:M\to M : D \text{ is an isometry such that } p\circ D = p\},$$
  is isomorphic to $\Z$, and
  \item[(3)] for every $x\in M_0$, there exists $\tilde{x}\in M$ such that $p^{-1}(x)=\{D(\tilde{x}): D\in\text{Deck}(M,p)\}$.
\end{enumerate}
Observe that every $D\in\text{Deck}(M,p)$ acts on $T^1M$ by its differential. We also denote this action by $D$. Despite $\Z$-coverings are geometrically infinite, in some cases we can say something about the dynamics of the geodesic flow. For instance, in the case presented below, the geodesic flow is ergodic with respect to the Bowen-Margulis measure (see for instance \cite{Rees}).

\begin{theorem}\label{thm:geom_inf} Let $M$ be a regular $\Z$-cover of a compact hyperbolic surface. Then,
\begin{enumerate}
  \item[(1)] The surface $M$ is geometrically infinite, and
  \item[(2)] The geodesic flow on $T^1M$ is conservative and ergodic with respect to the Bowen-Margulis measure.
\end{enumerate}
\end{theorem}

Theorem \ref{thm:geom_inf} implies that the fundamental group $\Gamma$ of a regular $\Z$-cover $M$ of a compact hyperbolic surface is of divergent type as consequence Theorem \ref{thm:pts}. This implies that the topological entropy of the geodesic flow on $M$ is 1. In more generality we can prove the following

\begin{theorem} Let $M=\widetilde{M}/\Gamma$ be a regular $\Z$-cover of a compact negatively curved  manifold. Then,
$$h_\infty(g)=\delta_\Gamma.$$
\end{theorem}
\begin{proof}
Let $(\mu_n)$ be a sequence of $(g_t)$-invariant probability measures on $T^1M$ converging vaguely to 0. By variational principle (see \cite{MR2097356}), we have
$$\limsup_{n\to\infty}h_{\mu_n}(g)\leq \delta_\Gamma,$$
so we only need to prove that for every $\varepsilon>0$ there exists a sequence $(\mu_n)$ of $(g_t)$-invariant probability measures such that $\mu_n\rightharpoonup 0$ and $h_{\mu_n}(g)\geq h_{top}(g)-\varepsilon$. Fix $\varepsilon>0$ and consider any $(g_t)$-invariant probability measure $\mu$ such that $h_{\mu}(g)>h_{top}(g)-\varepsilon$. By definition $\text{Deck}(M,p)$ is isomorphic to $\Z$, so there exists $D\in \text{Deck}(M,p)$ such that $<D>=\text{Deck}(p,M)$. The action of $D$ on $T^1M$ is also denoted by $D$. Let $\mu_n=(D^n)_\ast\mu$ be the image measure of $\mu$ by the map $D^n$. We claim that $(\mu_n)$ is the desired sequence of measures. On the other hand, for every $n\in\Z$, we have
\begin{equation}\label{eq:com}
D^n\circ g_t = g_t \circ D^n.
\end{equation}
It implies that $\mu_n$ is invariant by the action of the geodesic flow since
\begin{eqnarray*}
\mu_n(g_t(A))&=&\mu(D^{-n}(g_t(A)))=\mu(g_t(D^{-n}A))\\
&=&\mu(D^{-n}A)=\mu_n(A).
\end{eqnarray*}
Using again \eqref{eq:com} it follows that the map $D^n$ defines a measure-conjugation between $(T^1M,\mu,g)$ and $(T^1M,\mu_n,g)$, hence $h_\mu(g)=h_{\mu_n}(g)$. Let $\mathcal{D}\subset T^1M$ be a fundamental domain for the action of $\text{Deck}(M,p)$ over $T^1M$. If $K\subset T^1M$ is a compact set, then there exists a finite set $\mathcal{A}$ of integer numbers such that
$$K=\bigcup_{m\in\mathcal{A}}D^m(K\cap\mathcal{D}).$$
For $m\in\mathcal{A}$ define $K_m=D^m(K\cap\mathcal{D})$. By construction $\{D^n K_m:n\in\Z\}$ is a family of disjoint sets, so
$$\sum_{n\in\N}\mu_{n}(K_m)=\sum_{n\in\N}\mu(D^{-n}K_m) = \mu\left(\bigcup_{n\in\N}D^{-n}K_m\right)\leq 1.$$
In particular, the sequence $(\mu_{n}(K_m))_n$ converges to 0. Since
$$D^{-n}K=\bigsqcup_{m\in\mathcal{A}} D^{-n}K_m,$$
we necessarily have that $(\mu_{n}(K))_n$ converges to 0, which implies $\mu_n \rightharpoonup 0$. This concludes the proof of the theorem.
\end{proof}

\section{Final remarks}\label{final}

The measure of maximal entropy for the geodesic flow is known to exists if $M$ is compact or convex cocompact (i.e. the convex core is compact). 	In those cases the geodesic flow is modelled as a suspension flow over a Markov shift of finite type. An explicit formula for this measure was obtained by Sullivan in his pioneer work on the subject (in the hyperbolic case). When the nonwandering set is not compact, the existence of the measure of maximal entropy is more delicate. There are plenty of examples where the measure of maximal entropy simply does not exist. In the geometrically finite case Dal'bo, Otal and Peign\'e gave a characterization of the finitude of the Bowen-Margulis measure in terms of a modified Poincar\'e series running on the maximal parabolics subgroups of $\Gamma$.  A simple corollary of their theorem (which might be more practical than the theorem itself) is the following
\begin{theorem} Let $(M,g)$ be a geometrically finite Riemannian manifold with pinched negative sectional curvature. Assume that the derivatives of the sectional curvature are uniformly bounded. Suppose $\overline{\delta}_\P<\delta_\Gamma$. Then the geodesic flow of $M$ has a measure of maximal entropy.
\end{theorem}
\begin{proof}
Take a sequence $\{\mu_n\}_{n\ge 1}$ of invariant probability measures such that $h_{\mu_n}(g)\to h_{top}(g)$. Choose $\mu$ any vague limit of the sequence and observe that by using Theorem  \ref{thm:escape_mass} we can conclude $\mu$ is a probability measure. Moreover $h_\mu(g)=h_{top}$, as required.
\end{proof}

Equivalently if there is no measure of maximal entropy then $\overline{\delta}_\P=h_{top}(g)$. If we choose a sequence  $\{\mu_n\}_{n\ge 1}$ as above, then necessarily any vague limit must have cero mass. In other words
\begin{theorem}Let $(M,g)$ be a geometrically finite Riemannian manifold with pinched negative sectional curvature. Assume that the derivatives of the sectional curvature are uniformly bounded. If there is no measure of maximal entropy, then any sequence of measures approximating the topological entropy must converge to zero.
\end{theorem}

This is an interesting phenomena since the nonexistence of measure of maximal entropy simply means that the Bowen-Margulis measure is infinite. In this case the measures with high entropy must concentrate their mass in the cusps. It is an interesting question whether or not a sequence of measures whose entropies approximate the topological entropy converge to the Bowen-Margulis measure in a more \emph{suitable} sense or not.

We finish this paper with a small discussion about when we can weaken the H\"older regularity assumption of the potential. Let  $F:T^1M\to \R$ be a continuous function going to zero through the cusps of $M$. In particular if $\{\mu_n\}$ converges vaguely to $\mu$. Then
\begin{align*}
\int_K F d\mu_n + \sup\{|F(x)|: x\in K^c\}  &\ge \int_K F d\mu_n+\int_{K^c} F d\mu_n=\int Fd\mu_n,\\
\int_K Fd\mu+ \sup\{|F(x)|: x\in K^c\} &\ge \limsup_{n\to \infty} \int F d\mu_n,
\end{align*}
and therefore by taking the limit over compacts $K\subset T^1M$ we have
$$\int Fd\mu\ge \limsup_{n\to \infty} \int F d\mu_n.$$
The following easy consequence was already stated as Theorem \ref{existence_potentials} above. Here we emphazise the necessary conditions, in particular we prove the existence of equilibrium measures without any H\"older regularity assumption on the potential.
\begin{theorem} Under the hypothesis of Theorem \ref{thm:escape_mass}, for any continuous function $F:T^1M\to\R$ going to zero through the cusps of $M$, we have
$$\limsup_{n\to\infty} P(F,\mu_n) \leq \|\mu\|P(F,\mu/|\mu\|)+(1-\|\mu\|)\overline{\delta}_{\P}.$$
In particular, if $P(F)>\overline{\delta}_\P$, then there exists at least one equilibrium measure for the potential $F$.
\end{theorem}

\begin{remark} Without the H\"older assumption we can not identify the pressure of $F$ with the critical exponent of its Poincar\'e series and we can not ensure the uniqueness of the equilibrium state. It is an interesting question whether or not this weaker regularity assumption might allow to have more than one equilibrium state.
\end{remark}

\addcontentsline{toc}{chapter}{References}
\bibliography{biblio}
\bibliographystyle{amsalpha}

\end{document}